\documentclass{amsart}
\usepackage{latexsym,amsxtra,amscd,ifthen}
\usepackage{amsfonts}
\usepackage{verbatim}
\usepackage{amsmath}
\usepackage{amsthm, color}
\usepackage{amssymb}
\usepackage{url}
\usepackage[arrow,matrix]{xy}

\theoremstyle{plain}

\newtheorem{theorem}{Theorem}
\newtheorem{lemma}[theorem]{Lemma}
\newtheorem{proposition}[theorem]{Proposition}
\newtheorem{corollary}[theorem]{Corollary}
\newtheorem{conjecture}[theorem]{Conjecture}
\newtheorem{hypothesis}[theorem]{Hypothesis}

\numberwithin{theorem}{section}
\numberwithin{equation}{theorem}

\theoremstyle{definition}
\newtheorem{definition}[theorem]{Definition}
\newtheorem{setup}[theorem]{Setup}

\newtheorem{remark}[theorem]{Remark}
\newtheorem{question}[theorem]{Question}
\newtheorem*{question*}{Question}

\DeclareMathOperator{\RHom}{RHom}

\DeclareMathOperator{\End}{End}
\DeclareMathOperator{\Ext}{Ext}
\DeclareMathOperator{\Tor}{Tor}
\DeclareMathOperator{\Hom}{Hom}

\DeclareMathOperator{\GKdim}{GKdim}

\DeclareMathOperator{\gldim}{gldim}

\DeclareMathOperator{\Aut}{Aut}
\DeclareMathOperator{\gr}{gr}
\DeclareMathOperator{\Mod}{Mod}

\DeclareMathOperator{\uTor}{{\underline{Tor}}}
\DeclareMathOperator{\uExt}{{\underline{Ext}}}
\DeclareMathOperator{\uHom}{{\underline{Hom}}}
\def\Rpf{{\sf r}}
\def\Pty{{\sf p}}

\def\GKMod{\operatorname{Mod}}
\def\GKmod{\operatorname{mod}}
\def\mmod{\operatorname{mod}}
\def\QMod{\operatorname{QMod}}
\def\qmod{\operatorname{qmod}}

\def\GKgr{\operatorname{gr}}
\def\Gr{\operatorname{Gr}}
\def\QGr{\operatorname{QGr}}
\def\qgr{\operatorname{qgr}}
\newcommand{\inth}{\textstyle \int}

\begin{document}

\title[Pertinency and Quotient categories]
{Pertinency of Hopf actions and \\
Quotient categories of Cohen-Macaulay algebras}

\author{Y.-H. Bao, J.-W. He and J. J. Zhang}

\address{Bao: School of Mathematical Sciences,
Anhui University, Hefei, Anhui, 230601, China}

\email{baoyh@ahu.edu.cn, yhbao@ustc.edu.cn}

\address{He: Department of Mathematics,
Hangzhou Normal University,
Hangzhou Zhejiang 310036, China}
\email{jwhe@hznu.edu.cn}

\address{Zhang: Department of Mathematics, Box 354350,
University of Washington, Seattle, Washington 98195, USA}

\email{zhang@math.washington.edu}

\begin{abstract}
We study invariants and quotient categories of fixed subrings
of Artin-Schelter regular algebras under Hopf algebra actions.
\end{abstract}

\subjclass[2000]{Primary 16D90, 16E65, 16E10, Secondary 16B50, 16E35}


\keywords{Artin-Schelter regular algebra,
quotient category, Auslander theorem,
Hopf algebra action, pertinency, Cohen-Macaulay property.}


\maketitle


\setcounter{section}{-1}
\section{Introduction}
\label{xxsec0}

Artin-Schelter regular (or AS regular, for short) algebras
[Definition \ref{xxdef1.8}] play an important
role in noncommutative algebraic geometry. We study some classes of graded
algebras that are closely related to AS regular algebras. In particular,
we are interested in algebraic and homological properties of fixed subrings of
AS regular algebras under semisimple
Hopf algebra actions, or geometrically, properties of
quotient singularities of noncommutative spaces under quantum group actions.
Some homological properties of Hopf actions on AS regular algebras were studied
in \cite{KKZ1}. Finite dimensional Hopf actions on connected graded AS regular
algebras of global dimension two with trivial homological determinant were
classified in \cite{CKWZ1}. In this paper we further investigate Hopf actions
on AS regular algebras and the correspondence quotient singularities in the
noncommutative setting.

One of most interesting features in this area is the \emph{McKay correspondence}.
The classical McKay correspondence connects many different topics in mathematics
such as the finite subgroups of ${\rm{SU}}(2,{\mathbb C})$, Dynkin diagrams of
$ADE$ type, simple Lie algebras, Kleinian singularities, $G$-Hilbert schemes and
so on. See \cite{Bu, Le1, Le2} for very nice recent surveys and historical accounts.
One ingredient in the classical or commutative McKay correspondence is the
Auslander Theorem \cite{Au1,Au2} that states that, {\it if $G$ is a
finite small subgroup of ${\rm{GL}}({\mathbb C}^{\oplus n})$ acting on
$R:={\mathbb C}[x_1,\cdots,x_n]$ naturally, then there is a natural
isomorphism of algebras
\begin{equation}
\label{E0.0.1}\tag{E0.0.1}
R\ast G\cong\End_{R^G}(R).
\end{equation}}
Recall that a finite subgroup $G\subseteq {\text{GL}}(V)$, for a finite
dimensional vector space $V$, is called {\it small} if $G$ does not contain
a pseudo-reflection of $V$.

Different extensions and noncommutative versions of the McKay correspondence
have been proposed by several researchers \cite{Ch, CKWZ2, Mor2, Le2, Wem}
and many nice results have been proven. In noncommutative
algebraic geometry, AS regular algebras are correct noncommutative analogues
of the commutative polynomial rings. One key ingredient in
noncommutative McKay correspondence is the Auslander Theorem
when $R$ is replaced by a noncommutative AS regular algebra.

It is desirable to prove the Auslander Theorem in
a general noncommutative setting. Some partial results can be found
in \cite{BHZ, CKWZ2, HVZ2, MU} and so on, under some extra hypotheses.
One of the main goals of this paper is to prove a noncommutative version
of the Auslander theorem that recovers most of the previous results
and to provide some conjectures in this research direction.

Let $R$ be an ${\mathbb N}$-graded algebra over a base field $\Bbbk$.
We are taking one step further by replacing the group $G$ in \eqref{E0.0.1}
by a Hopf algebra $H$.
Let $H$ be a nontrivial semisimple Hopf algebra acting on $R$ homogeneously
and inner faithfully [Definition \ref{xxdef3.9}]. Let $R\# H$ denote
the smash product \cite[Definition 4.1.3]{Mon}. Since $H$ is semisimple,
the left and the right integrals of $H$ coincide.
Denote by $\inth$ the integral of $H$ such that $\varepsilon(\inth)=1$.
The main {\it dimension function} we are going to use here is the
Gelfand-Kirillov
dimension (also denoted by GK-dimension) -- see \cite{KL} or
Definition \ref{xxdef1.1}. A new invariant we
introduce in this paper is the following.

\begin{definition}
\label{xxdef0.1} Retain the above notation and suppose that
$\GKdim R<\infty$. The {\it pertinency}
of the $H$-action on $R$ is defined to be
$$\Pty(R,H)=\GKdim R-\GKdim ((R\# H)/I)$$
where $I$ is the 2-sided ideal of $R\# H$ generated by
$1\# \inth$.
\end{definition}

Fix an $R$ and an $H$, different $H$-actions on $R$ result in different
pertinency. This means that $\Pty(R,H)$ is an invariant of the $H$-action
on $R$, not just the pair $(R,H)$. When $H$ is the group algebra
$\Bbbk G$, we write
$\Pty(R,H)$  as $\Pty(R,G)$. If the $H$-action on $R$ is not
inner faithful, then it is easy to check that $\Pty(R,H)=0$.
If $H$ acts on a domain $R$ inner faithfully and $R\# H$ is prime, it is
automatic that $\Pty(R,H)\geq 1$ (see Lemma \ref{xxlem3.10}).
Note that $\Pty(R,H)=\GKdim R$ (or equivalently, $\GKdim ((R\# H)/I)
=0$) if and only if $R/R^H$ is a right $H^*$-dense Galois extension in
the sense of \cite{HVZ2}.

\begin{definition}
\label{xxdef0.2} Let $R$ be a noetherian algebra with finite GK-dimension. 
Suppose $H$ is a Hopf algebra with integral
$\inth$ acting on $R$ inner faithfully,
and let $M$ be a finitely generated right $R$-module.
\begin{enumerate}
\item[(1)]
The {\it grade} of $M$ is defined
to be
$$j(M):=\min\{i\mid \Ext_R^i(M,R)\neq0\}.$$
If $\Ext_R^i(M,R_R)=0$ for all $i$, then we say $j(M)=\infty$.
\item[(2)]
We say that the $H$-action on $R$ is \emph{homologically small}
if $j(R\# H/(I))\geq 2$, where $I$ is the 2-sided ideal
generated by $1\# \inth$.
\end{enumerate}
\end{definition}

In part of the paper, we use Cohen-Macaulay property [Definition \ref{xxdef1.4}(3)]
which is a homological property in the noncommutative setting
that is intimately connected with
the AS regularity. For example, all well-studied noetherian AS
regular algebras have been proven to be Cohen-Macaulay and it is
conjectured that the Cohen-Macaulay property is a consequence
of the AS regularity.

{\bf Setup:} \emph{For the rest of introduction, let $R$ denote a noetherian,
connected graded AS regular, Cohen-Macaulay
algebra of GK-dimension $\geq 2$.}


In the commutative (graded) case, homological smallness
is equivalent to the (usual) smallness \cite[Lemma 7.2]{BHZ}.
Our first result connects the noncommutative Auslander theorem
with the pertinency.

\begin{theorem}
\label{xxthm0.3}
Let $H$ be a semisimple Hopf algebra acting on $R$ inner faithfully
and homogeneously. Then the following are equivalent.
\begin{enumerate}
\item[(1)]
$\Pty(R,H)\geq 2$.
\item[(2)]
The $H$-action on $R$ is homologically small.
\item[(3)]
There is a natural isomorphism of algebras $R\# H\cong \End_{R^H}(R)$.
\end{enumerate}
\end{theorem}

In Theorems \ref{xxthm0.3} and \ref{xxthm0.6} below,
the Cohen-Macaulay property is crucial while the AS
regularity is hidden in the background. Then, in
Theorem \ref{xxthm0.7}, the AS regularity is definitely
needed.

Theorem \ref{xxthm0.3} does not recover the classical Auslander
theorem, however, it extends most of previous results concerning
the noncommutative Auslander theorem in \cite{CKWZ3, HVZ2, MU}.
It is easier to manipulate the condition on $\Pty(R,H)$ than to
prove the isomorphism in Theorem \ref{xxthm0.3}(3) directly, see
\cite{BHZ}. This is one of the advantages of Theorem \ref{xxthm0.3}.

By the above theorem, checking the hypothesis on the pertinency
``$\Pty(R,H)\geq 2$'' becomes essential. In some cases, the
inequality $\Pty(R,H)\geq 2$ is checkable -- see
Theorem \ref{xxthm0.5} and \cite{BHZ}.
We do not have a general algorithm of calculating this
number. However, the following conjecture is expected.

\begin{conjecture}
\label{xxcon0.4}
Retain the hypotheses as in Theorem \ref{xxthm0.3}.
Suppose the $H$-action on $R$ has trivial homological
determinant in the sense of
\cite[{\rm{Definition 3.3}}]{KKZ1}. Then $\Pty(R,H)\geq 2$,
or equivalently, the $H$-action on $R$ is homologically
small.
\end{conjecture}

When $R$ is the commutative polynomial ring and $H$ is a finite
group algebra, this conjecture holds as a consequence of the
classical Auslander theorem. In the noncommutative setting, this
has been verified in some special cases, but the conjecture is
open in general. If Conjecture \ref{xxcon0.4} holds, then, by
combining with Theorem \ref{xxthm0.3}, we have
the following version of the Auslander theorem:
\emph{Retain the hypotheses as in Theorem
{\rm{\ref{xxthm0.3}}} and further assume that
the $H$-action on $R$ has trivial homological determinant. Then
$R\# H\cong \End_{R^H}(R)$ naturally.}

Note that the aim of paper \cite{CKWZ2} is to establish a noncommutative McKay
correspondence for Hopf algebra actions on AS regular algebras of
dimension two, and one of the main results there is to verify
Conjecture \ref{xxcon0.4} when $\gldim R=2$.

As a consequence of Theorem \ref{xxthm0.3}, we have the following
result. For any graded algebra $A$, let $\Aut_{gr}(A)$
denote the group of all graded algebra automorphisms of $A$. Suppose
$n\geq 2$ and let $R:=\Bbbk_{-1}[x_1,\cdots,x_n]$ be the $(-1)$-skew polynomial
ring generated by $x_1,\cdots,x_n$ and subject to the
relations $x_ix_j=-x_jx_i$ for all $i\neq j$. Then $\Aut_{gr}(R)=
(\Bbbk^{\times})^n\rtimes S_n$ \cite[Lemma 1.12(1)]{KKZ2}.

\begin{theorem}
\label{xxthm0.5}
Let $\sigma$ be the automorphism of $R:=\Bbbk_{-1}[x_1,\cdots,x_n]$
determined by sending $x_i\mapsto x_{i+1}$
for all $1\leq i\leq n-1$ and $x_n\mapsto x_1$. Let $G$ be the group
generated by $\sigma$. Then the homological determinant of the
$G$-action on $R$ is trivial and $\Pty(R,G)\geq 2$. As a consequence,
$R\ast G\cong\End_{R^G}(R).$
\end{theorem}

Theorem \ref{xxthm0.5} answers a question of Ellen Kirkman
\cite{Ki}, see also Question \ref{xxque5.9}.
In contrast, when $R$ is a commutative polynomial ring
$\Bbbk[x_1,\cdots,x_n]$ and $G$ is generated by the permutation 
sending $x_i\mapsto x_{i+1}$ for all $1\leq i\leq n-1$
and $x_n\mapsto x_1$, then
$$\Pty(R,G)=\begin{cases} 1 & n=2, \\ \geq 2 & n>2.\end{cases}$$
So, in the commutative case, $G$ is not (homologically) small
when $n=2$, see \cite[Examples 7.3 and 7.4]{BHZ} for further comments.

Going back to a general (graded) algebra $A$,
we write $\Gr A$ for the category of right graded $A$-modules, and
$\gr A$ for the subcategory of finitely generated right graded $A$-modules. For
any integer $n$,  let $\Gr_n A$ (respectively, $\gr_n A$) be the subcategory of
$\Gr A$ (respectively, $\gr A$) consisting of objects with GK-dimension not larger
than $n$. Define the following quotient categories
$$\QGr_n B:=\frac{\Gr B}{\Gr_n B}\qquad \text{ and }\qquad
\qgr_n B:=\frac{\gr B}{\gr_n B}.$$
When $n=0$, the quotient category $\QGr_0 B$ (respectively, $\qgr_0 B$) is
just $\QGr B$ (respectively, $\qgr B$) defined in
noncommutative projective geometry
\cite{AZ} or ${\text{Tails}} \; B$ (respectively, ${\text{tails}} \; B$) used
by other authors \cite{MU, Ue, Mor1, Mor2}.

Our first result relating the pertinency and the
quotient category is the following.

\begin{theorem}
\label{xxthm0.6}
Retain the hypotheses as in Theorem {\rm{\ref{xxthm0.3}}}.
Let $\alpha$ be an integer $\leq \Pty(R,H)$ and assume  $\GKdim R=n$. Then there is an
equivalence of categories
$$\qgr_{n-\alpha} R^H\cong \qgr_{n-\alpha} R\# H.$$
\end{theorem}

Since $R$ is AS regular, $R\# H$ has finite global
dimension. If, in addition, $R$ is PI (namely, $R$ satisfies a
polynomial identity), then the quotient category $\qgr_{n-\alpha} R\# H$
has finite global dimension. Under the extra PI hypothesis, by the above
theorem, $\qgr_{n-\alpha} R^H$ has finite global dimension. When a
connected graded algebra $A$ is PI, it is reasonable to define the
dimension of the singular locus of $A$ to be
$$\dim_{sing}(A)=\min \{ i\mid \qgr_i A \; {\text{has finite global
dimension}}\}.$$
Using this definition and assuming $R$ is PI, we can bound the
$\dim_{sing}$ of the fixed subring $R^H$ by
\begin{equation}
\label{E0.6.1}\tag{E0.6.1}
\dim_{sing}(R^H)\leq n-\Pty(R,H)=\gldim R-\Pty(R,H).
\end{equation}
In particular, when $\Pty(R,H)=\gldim R$,
$R^H$ has graded isolated singularities in
the sense of Ueyama. 
Recall from \cite[Definition 2.2]{Ue} a noetherian connected graded algebra $A$ is a {\it graded isolated singularity} if
${\rm tails}A(=\qgr A)$ has finite global dimension. We give new examples of
noncommutative graded isolated singularities in Proposition \ref{xxpro5.4}.
Following the idea of the noncommutative crepant resolutions \cite{Le2},
$R\# H$ should be viewed a resolution of singularities of
$R^H$. Some ideas in this paper are motivated by recent papers
\cite{CKWZ1, CKWZ2, Le2, Mor1, Mor2, MU, HVZ2, Ue}.

Another result concerning quotient categories is a version of BGG
correspondence. Let $R$ be a Koszul AS regular algebra and $E$ the graded algebra
$$E:=\bigoplus_{i=0}^n \Ext_{R\# H}^i(S,S)$$
where $S$ is a direct sum of all non-isomorphic graded simple right
modules over $R\# H$ concentrated in degree zero. Since $R$ is Koszul,
so is $R\# H$. As a consequence, $E$ is Koszul. (All these are
${\mathbb N}$-graded and locally finite.) Since $R$ is AS regular,
$E$ is Frobenius. Then the stable module category $\underline{\gr}\; E$
is a triangulated category. Let $M$ be a finite dimensional right
graded $E$-module. Take a minimal graded projective resolution
of $M$ as follows
$$\cdots\longrightarrow P^{-n}\longrightarrow\cdots\longrightarrow
P^0\longrightarrow M\longrightarrow 0.$$
The {\it complexity}  of $M$, denoted by $c(M)$,  is the least
integer $\lambda\ge0$ such that $\dim P^{-n}<a\, n^{\lambda-1}$
for almost all $n$, where $a>0$ is a fixed number. The complexity
of a module over a finite dimension algebra is an important
invariant in studying the representations of finite dimensional
algebras \cite{Ca,CDW,GLW}. Given any nonnegative integer $\beta$,
let $\mathcal{C}_\beta$ be the full subcategory of $\gr E$
consisting of objects $M$ with $c(M)\leq \beta$.

\begin{theorem}
\label{xxthm0.7}
Retain the hypotheses as in Theorem {\rm{\ref{xxthm0.3}}} and
assume further that $R$ is Koszul and that $\Bbbk$ is algebraically
closed. Let $\beta$ be a nonnegative integer.
\begin{enumerate}
\item[(1)]
There is a canonical anti-equivalences of triangulated categories
$$D^b(\qgr_\beta R\# H)\cong \underline{\gr}\; E/\mathcal{C}_\beta.$$
\item[(2)]
For any $\beta\geq \GKdim R-\Pty(R,H)$, there is an canonical
anti-equivalence of triangulated categories
$$D^b(\qgr_{\beta} R^H)\cong \underline{\gr}\; E/\mathcal{C}_{\beta}.$$
\end{enumerate}
\end{theorem}

When $H$ is a group algebra, one can define another invariant associated
to the $G$-action on $R$. Let now $R$ be as in Theorem \ref{xxthm0.7}.
Then the global dimension and GK-dimension of $R$ are the same. For any
graded algebra automorphism $g$ of
$R$, one can define the trace function 
$Tr(g,t)=\colon \sum\limits_{i=0}^\infty Tr(g|_{R_i})t^i$ as in \cite[Section 1]{KKZ2}.

\begin{definition}
\label{xxdef0.8}
Let $g$ be a graded algebra automorphism of $R$ of finite order
and let $G$ be a finite subgroup of $\Aut_{gr}(R)$.
\begin{enumerate}
\item[(1)]
The {\it reflection number} of $g$ is defined to be
$$\Rpf(g)=\GKdim R-{\text{the order of the pole of $Tr(g,t)$ at $t=1$.}}$$
\item[(2)]
\cite[Definition 2.2]{KKZ3}
$g$ is called a {\it quasi-reflection} if $\Rpf(g)=1$.
\item[(3)]
\cite[Definition 3.7(b)]{KKZ4}
$g$ is called a {\it quasi-bireflection} if $\Rpf(g)=2$.
\item[(4)]
The {\it reflection number} of $G$-action on $R$ is defined to be
$$\Rpf(R,G)=\min \{ \Rpf(g)\mid 1\neq g\in G\}.$$
\item[(5)]
$G$ is called {\it conventionally small} if $\Rpf(R,G)\geq 2$, or equivalently, $G$
does not contain any quasi-reflection.
\end{enumerate}
\end{definition}

Similar to $\Pty(R,G)$, $\Rpf(R,G)$ is an invariant of
the $G$-action on $R$, not of the pair $(R,G)$.
The relation between $\Pty(R,G)$ and $\Rpf(R,G)$ is not
straightforward, and what we are hoping for is the
following conjectural inequality.

\begin{conjecture}
\label{xxcon0.9}
Retain the hypotheses as in Theorem {\rm{\ref{xxthm0.7}}}.
Let $G$ be a finite subgroup of $\Aut_{gr}(R)$. Then
\begin{equation}
\label{E0.8.1}\tag{E0.8.1}
\Pty(R,G)\geq \Rpf(R,G).
\end{equation}
\end{conjecture}

This conjecture can be verified for some special cases
[Theorem \ref{xxthm5.7}(4)]. If it holds, then Conjecture
\ref{xxcon0.4} follows when $H$ is a group algebra
because $\Rpf(R,G)\geq 2$ when the $G$-action has trivial
homological determinant. Indeed, if $g\in G$ is a 
quasi-reflection, then $Tr(g,t)=\frac{1}{(1-t)^{n-1}(1-\xi t)}$ 
where $n=\GKdim R$ and $\xi\neq 1$. Since the homological 
determinant of $G$-action is trivial, we obtain that $\xi=1$, 
a contradiction. Therefore $G$ contains no quasi-reflections, 
that is, $\Rpf(R,G)\ge2$. If Conjecture \ref{xxcon0.9}
holds, then, by combining with Theorem \ref{xxthm0.3},
we obtain the following: \emph{Retain the hypotheses
as in Theorem {\rm{\ref{xxthm0.7}}} and further assume that
the $G$ is conventionally small. Then
$R\ast G\cong \End_{R^G}(R)$ naturally.}

Combining \eqref{E0.6.1} with \eqref{E0.8.1}
we have another conjecture
\begin{equation}
\label{E0.8.2}\tag{E0.8.2}
\dim_{sing}(R^G)\leq \gldim R-\Rpf(R,G).
\end{equation}

The advantage of this conjectural inequality is that, in contrast to $\Pty(R,G)$,
the invariant $\Rpf(R,G)$ is relatively easy to calculate. Note that
\eqref{E0.8.1} can be verified when $R$ is AS regular of global
dimension two following the classification of all such $G$-actions
in \cite{CKWZ1}.
When the $G$-action on $R$ has trivial homological determinant,
then $\Pty(R,G)$ (respectively, $\Rpf(R,G)$) would be either
2 or $3$ and $\dim_{sing}(R^G)$ would be either 0 or 1. Understanding
all possibilities of $\{\Pty(R,G), \Rpf(R,G), \dim_{sing}(R^G)\}$
in dimension three should be an interesting project.

This paper is organized as follows. We provide background material
in Section 1. In Section 2, we prove a Morita type theorem for
the equivalences of the quotient categories. In Section 3, we
consider semisimple Hopf actions on graded Cohen-Macaulay algebras.
We prove a generalized version of Theorem \ref{xxthm0.3}, and then
give a proof of Theorem \ref{xxthm0.6}. In Section 4, we
establish a version of BGG correspondence between the derived
categories of quotient categories and the stable categories of
some Frobenius categories. In particular, we obtain Theorem
\ref{xxthm0.7}. In Section 5, we compute the pertinency of
finite group actions on some special algebras, and
give a proof of Theorem \ref{xxthm0.5}.

\section{Preliminaries}
\label{xxsec1}

Throughout let $\Bbbk$ be a base ring that is a noetherian commutative
domain.  All algebraic objects are over $\Bbbk$. An algebra means an
associative $\Bbbk$-algebra that is projective as a $\Bbbk$-module.
In Sections 4 and 5 (and part of Section 3) we further assume that
$\Bbbk$ is a field. Let $B$ be a (left and right) noetherian algebra.
Usually we are working with right $B$-modules. We write $\Mod B$ for
the category of all right $B$-modules, and $\mmod B$ for the full
subcategory of all finitely generated (or f.g. for short)
right $B$-modules.

Several results in this paper concern quotient categories of
the module categories defined  via a dimension function.
For the purpose of this paper, a function
$$\partial: \Mod B \longrightarrow {\mathbb R}\cup \{\pm \infty\}$$
is called a {\it dimension function} if, for all f.g. $B$-modules $M$,
\begin{equation}
\label{E1.0.1}\tag{E1.0.1}
\partial(M)\geq \max\{\partial(N),\partial(M/N)\},
\end{equation}
whenever $N$ is a submodule of $M$. The
$\partial$ is called an {\it exact dimension function}
if for all f.g. $B$-modules $M$,
\begin{equation}
\label{E1.0.2}\tag{E1.0.2}
\partial(M)=\max\{\partial(N),\partial(M/N)\}
\end{equation}
whenever $N$ is a submodule of $M$. If $M$ is not f.g.,
we extend the definition of $\partial$ by
\begin{equation}
\label{E1.0.3}\tag{E1.0.3}
\partial(M)=\sup\{\partial(N)\mid {\text{ for all f.g.
$B$-submodules $N\subseteq M$}}\}.
\end{equation}
By using \eqref{E1.0.3}, equation \eqref{E1.0.2} holds for all
$B$-modules $M$ as well. Our definition of a dimension
function is weaker than the definition given in
\cite[6.8.4]{MR}, however, \eqref{E1.0.2} is the only property
we will use.

For a $\Bbbk$-module $V$, $\dim V$ denotes $\dim_{Q} (V\otimes_{\Bbbk} Q)$
where $Q$ is the fraction field of $\Bbbk$.
One special dimension function is the following

\begin{definition}
\label{xxdef1.1} \cite[Definition 2.1]{KL}
Let $B$ be an algebra and $M$ a right $B$-module.
\begin{enumerate}
\item[(1)]
The {\it Gelfand-Kirillov dimension} or
{\it GK-dimension} of $B$
is defined to be
$$\GKdim  B=\sup_V \{\overline{\underset{n\to\infty}\lim}
\log_n (\dim V^n) \mid
 {\text{f.g. $\Bbbk$-submodules}}\; V \subseteq B\}. $$
\item[(2)]
The {\it Gelfand-Kirillov dimension} or {\it GK-dimension}
of $M$ is defined to be
$$\GKdim M=\sup_{V,W} \{\overline{\underset{n\to\infty}\lim}
\log_n (\dim W V^n) \mid
{\text{f.g. $\Bbbk$-submodules}}\;V \subseteq B, W\subseteq M\}. $$
\end{enumerate}
\end{definition}

See \cite{KL} and Remark \ref{xxrem1.5} for basic properties
of the GK-dimension.

In a major part of the paper the dimension function is chosen to
be the Gelfand-Kirillov dimension. However,
a general dimension function will be useful when we deal with the
local setting in the sequel \cite{BHZ} in which
case the GK-dimension is usually
infinite and does not work well.

From now on we fix an exact dimension function $\partial$.
Given an integer $n\ge0$, let $\GKMod_n B$ denote the full subcategory
of $\Mod B$ consisting of right $B$-modules with $\partial$-dimension
less than or equal to $n$, and let
$$\GKmod_n B=\GKMod_n B\; \bigcap\; \mmod B.$$
Since $\partial$ is {\it exact} on right $B$-modules,
$\GKMod_n B$ (respectively, $\GKmod_n B$) is a Serre subcategory of $\Mod B$
(respectively, $\mmod B$). Hence it makes sense to define the quotient categories:
$$\QMod_n B:=\frac{\Mod B}{\GKMod_n B},\quad
{\text{and}}\quad \qmod_n B:=\frac{\mmod B}{\GKmod_n B}.$$
We denote the natural projection functor by
$$\pi:\Mod B\longrightarrow \QMod_n B.$$
For $M\in \Mod B$, we will write $\mathcal{M}$ for the object $\pi(M)$
in $\QMod_n B$. The hom-set of the objects in the quotient category is
defined by
$$\Hom_{\QMod_n B}(\mathcal{M},\mathcal{N})=
\underset{\longrightarrow}\lim\Hom_B(M',N')$$
for $M, N\in \Mod B$, where $M'$ is a submodule of $M$ such that
$\partial (M/M')\leq n$, $N'=N/T$ for some submodule $T$ with
$\partial (T)\leq n$, and the direct limit runs over all the
pairs $(M', N')$ with these properties.

\begin{definition}
\label{xxdef1.2}
Let $A$ and $B$ be noetherian algebras and
$\partial$ be an exact dimension function that is defined on right
$A$-modules and on right $B$-modules. Let $n$ and $i$ be nonnegative
integers. Let $_AM_B$ denote one (or any) bimodule which is f.g. both as a
left $A$-module and as a right $B$-module.
\begin{enumerate}
\item[(1)]
We say $\partial$ {\it satisfies $\gamma_{n,i}(M)$} if for any
$N\in \mmod_{n} A$, $\Tor^A_j(N, M)\in \mmod_n B$
for all $0\leq j\leq i$.
\item[(2)]
We say $\partial$ {\it satisfies $\gamma_{n,i}$} if it satisfies
$\gamma_{n,i}(M)$ for all $M$ given as above.
\end{enumerate}
\end{definition}

In the next section we will be particularly interested in
the $\gamma_{n,1}$ property.

\begin{lemma}
\label{xxlem1.3}
Let $A$ and $B$ be noetherian algebras such that $\partial$ is an exact
dimension function on both right $A$-modules and right $B$-modules.
Assume that a bimodule ${}_AM_B$ is f.g. as a left $A$ and as a right
$B$-module and that $\partial$ satisfies $\gamma_{n,1}(M)$.
\begin{enumerate}
\item[(1)]
The functor $-\otimes_A M$ induces a functor
$$-\otimes_\mathcal{A} \mathcal{M}: \qmod_n A\longrightarrow\qmod_n B.$$
\item[(2)]
The functor $-\otimes_A M$ induces a functor
$$-\otimes_\mathcal{A} \mathcal{M}: \QMod_n A\longrightarrow\QMod_n B.$$
\end{enumerate}
\end{lemma}

\begin{proof}
We first prove (2).

For $K\in\GKMod_n A$, it is a direct limit of f.g.
submodules. Since the direct limit is an exact functor and, by assumption,
$\partial$ is exact and satisfies $\gamma_{n,1}(M)$,
we obtain that $\partial(K\otimes_A M)\leq n$ and
that $\partial(\Tor_1^A(K,M))\leq n$ for all $K\in\GKMod_n A$.
Let $f:U\to V$ be a
morphism in $\Mod A$ such that $\ker f$ and co$\ker f$ are in $\GKMod_n B$.
By the right exactness of $-\otimes_A M$, the cokernel of the right $B$-module
morphism
$$f\otimes_AM:U\otimes_AM\to V\otimes_AM$$
is in $\GKMod_n B$. Write $f$ as the composition
$$f=h\circ g : U\overset{g}\twoheadrightarrow W\overset{h}\hookrightarrow V.$$
Since $\partial(\Tor_1^A(\text{coker}(f),M))\leq n$, by the long exact sequence
associated to $-\otimes_A M$,
$\partial(\ker (h\otimes_A M))\leq n$. Let $X=\ker f=\ker g$. Then
$\partial(X)\leq n$. We have an exact sequence
$$X\otimes_A M\longrightarrow U\otimes_A M\overset{g\otimes_A M}
\longrightarrow W\otimes_A M\longrightarrow 0$$
which induces the exact sequence
$$X\otimes_AM\longrightarrow \ker(f\otimes_AM)\longrightarrow \ker(h\otimes_AM)
\longrightarrow 0.$$
Since $\partial$ is exact, we see
$$\partial(\ker(f\otimes_AM))
\leq \max\{\partial(X\otimes_A M), \partial(\ker (h\otimes_A M))\}
\leq n.$$
Hence $f\otimes_AM$ has its kernel and cokernel in $\GKMod_n B$.
Therefore, $-\otimes_AM$ induces a functor
$$-\otimes_\mathcal{A}\mathcal{M}:\QMod_n A\longrightarrow \QMod_n B.$$
So we have proven part (2).

Since $M$ is f.g. in both sides, this functor restricts to
$$-\otimes_\mathcal{A}\mathcal{M}:\qmod_n A\longrightarrow \qmod_n B.$$
Therefore (1) follows.
\end{proof}

We now recall a definition.

\begin{definition}
\label{xxdef1.4}
Let $B$ be an algebra and $M$ a right $B$-module.
\begin{enumerate}
\item[(1)]
\cite[Definition 0.4]{ASZ1}
Let $\partial$ be a dimension function. We say $B$ is
{\it $\partial$-Cohen-Macaulay} (or {\it $\partial$-CM}) if
$\partial (B)=d\in\mathbb{N}$,
and
$$j(M)+\partial (M)=\partial (B),$$
for every f.g. right $B$-module $M\neq 0$.
\item[(2)]
If $B$ is $\GKdim$-Cohen-Macaulay, we just
say it is {\it Cohen-Macaulay} or {\it CM}.
\item[(3)]
Let $\partial$ be an exact dimension function and
$\partial(B)=d\in {\mathbb N}$. Let $s$ be a nonnegative
integer. We say $B$ is $\partial$-CM$(s)$ if, for every
f.g right $B$-module $M$,
$${\text{$\partial(M)\leq d-s$ $\Longleftrightarrow$ $j(M)\geq s$.}}$$
We just say $B$ is CM($s$) if it is GKdim-CM($s$).
\end{enumerate}
\end{definition}

The CM property together with the Artin-Schelter regularity
and the Auslander property (another homological property)
has been studied in the noncommutative setting in
\cite{ASZ1, ASZ2, SZ, YZ, Zh}.

\begin{remark}
\label{xxrem1.5}
In parts (4,5), assume that $\Bbbk$ is a field.
\begin{enumerate}
\item[(1)]
The dimension function $\partial$ is exact on $\partial$-CM algebras
\cite[p. 3]{ASZ2}.
\item[(2)]
$\GKdim$ is exact on filtered algebras such that the associated
graded algebras are locally finite and noetherian
\cite[Theorem 6.14]{KL}.
\item[(3)]
$\GKdim$ satisfies $\gamma_{n,0}$ for all $n$. Namely,
for any $N\in \GKmod_n A$, we have $\GKdim N\otimes_A M\leq n$,
assuming ${}_AM_B$ is f.g. in both sides.
This follows from \cite[Proposition 5.6]{KL} by a slight modification
of the proof.
\item[(4)] For every f.g. module $N$,
it is well-known that $\GKdim N=0$ if and only if $N$ is finite
dimensional over $\Bbbk$. Using this fact, one can easily check that
$\GKdim \Tor_i^A(N,M)=0$ if $\GKdim N=0$ and ${}_AM_B$ is f.g.
in both sides. This means that $\GKdim$ satisfies
$\gamma_{0,i}$ for all $i$.
\item[(5)]
Take $\partial=\GKdim$.
As a consequence of part (4) and Lemma \ref{xxlem1.3}(2),
the functor
$$-\otimes_\mathcal{A}\mathcal{M}:\QMod_0A\longrightarrow \QMod_0B$$
is always well-defined.
\end{enumerate}
\end{remark}

Let $B=\bigoplus_{n\ge 0} B_n$ be a noetherian graded
algebra. In this paper, a graded algebra means an ${\mathbb N}$-graded
and a graded module means ${\mathbb Z}$-graded. We call
$B$ {\it locally finite} if each $B_n$ has finite rank over $\Bbbk$.
We write $\Gr B$ for the category of
right graded $B$-modules, and $\gr B$ for the subcategory of
f.g. right graded $B$-modules. Let $\partial$ be an
exact dimension function on graded right $B$-modules. Similar to
the nongraded case, we write $\Gr_n B$ (respectively, $\gr_n B$) for the
subcategory of $\Gr B$ (respectively, $\gr B$) consisting of objects
with $\partial$-dimension not larger than $n$. Similar to
the definitions given in the introduction, we have quotient
categories
$$\QGr_n B:=\frac{\Gr B}{\Gr_n B}\qquad \text{ and }\qquad
\qgr_n B:=\frac{\gr B}{\gr_n B}.$$
We also use $\pi:\Gr B\to \QGr_n B$ for the natural projection
functor, and for $M\in \Gr B$, we write $\mathcal{M}:=\pi(M)$.

If $M$ is a f.g. graded right module over a
noetherian  locally finite graded algebra $B$,
then its GK-dimension can be computed by \cite[(E7)]{Zh}
\begin{equation}
\label{E1.5.1}\tag{E1.5.1}
\GKdim M
=\underset{k\to \infty}{\overline{\lim}}\log_k\sum_{j\leq k}\dim (M_j).
\end{equation}

Let $M$ and $N$ be right $A$ modules. As usual, we write 
$$\underline{\Hom}_ A(M, N)=\oplus_{j\in \mathbb{Z}}\Hom_{\Gr A}(M, N(j)),$$ 
where $\Hom_{\Gr A}(M, N)$ is the space of all the 
graded $A$-module homomorphisms preserving the gradings. 
If $N$ is a graded $A$-$B$-bimodule, the tensor product 
$M\otimes_A N$ is a right graded $B$-module by setting 
the $j$th component $(M\otimes_A N)_j$ to be the subspace 
of $M\otimes_AN$ spanned by elements $x\otimes_A y$, 
where $x\in M_k$ and $y\in N_{j-k}$ for all $k\in\mathbb{Z}$.
The $i$th derived functor of the $\underline{\Hom}_A(M,N)$ in the
graded category (respectively, $M\otimes_A N$)  is
denoted by $\uExt^i_A(M,N)$ (respectively, $\uTor^A_i(M,N)$).

\begin{lemma}
\label{xxlem1.6}
Let $A$ and $B$ be noetherian locally finite graded
algebras. Then  $\GKdim$ on graded modules satisfies $\gamma_{n,i}$
for all $n$ and $i$.
\end{lemma}

\begin{proof} Let $M$ be a graded $A$-$B$-bimodule such that
$M$ is f.g. on both sides. We need to show that,
for every f.g. graded right $A$-module $N$ with $\GKdim N\leq n$, we
have $\GKdim\uTor_i^A(N,M)\leq n$ for all $i\ge 0$.

Since $N$ is f.g., $\uTor_i^A(N,M)$ is
a graded f.g. $B$-module for each $i\ge 0$. Let
$$\cdots\longrightarrow F_n\longrightarrow\cdots\longrightarrow F_1
\longrightarrow F_0\longrightarrow M\longrightarrow 0$$
be a graded free resolution of the left graded $A$-module $M$. Since
$_A M$ is f.g., we can choose each $F_i$ to be
f.g.. By definition, $\uTor_i^A(N,M)$ (forgetting the
right $B$-module structure) is a subquotient of $N\otimes_A F_i$.
Temporarily write $T:=\uTor_i^A(N,M)$ and $\overline{N}:=N\otimes_A F_i$.
Since $T$ is f.g. as a right graded $B$-module and, as a
graded $\Bbbk$-module, $T$ is a subquotient of $\overline{N}$,
by \eqref{E1.5.1}, we have
$$\begin{aligned}
\GKdim T
&=\underset{k\to \infty}{\overline{\lim}}\log_k\sum_{j\leq k}\dim (T_j)\\
&\leq\underset{k\to \infty}{\overline{\lim}}\log_k\sum_{j\leq k}\dim
(\overline{N}_j)\\
&\textcolor[rgb]{1.00,0.00,0.00}{\leq} 
\underset{k\to \infty}{\overline{\lim}}\log_k\sum_{j\leq k}\dim (N_j)\\
&\leq \GKdim N =n.
\end{aligned}
$$
\end{proof}

\begin{lemma}
\label{xxlem1.7}
Let $A$ and $B$ be noetherian locally finite graded
algebras. Consider $\partial:=\GKdim$. For any graded $A$-$B$-bimodule
$M$ that is f.g. on both sides, a graded version of
Lemma {\rm{\ref{xxlem1.3}}} holds. Namely, for any $n\ge 0$, we have 
functors:
$$-\otimes_\mathcal{A} \mathcal{M}: \QGr_n A\longrightarrow\QGr_n B,
\qquad \text{and}\qquad
-\otimes_\mathcal{A} \mathcal{M}: \qgr_n A\longrightarrow\qgr_n B.$$
\end{lemma}

\begin{proof} By Lemma \ref{xxlem1.6}, $\partial$ satisfies
$\gamma_{n,i}(M)$ for all $n$ and $i$ and all $M$. The assertion
follows from a graded version of the proof of Lemma \ref{xxlem1.3}.
\end{proof}

To conclude this section we recall and introduce a couple of more definitions.

The notion of Artin-Schelter Gorensteinness was generalized to the
nonconnected graded case by several authors. We are going to use the
following version. Let $A$ be a locally finite graded algebra such that
$A_0$ is semisimple.

\begin{definition}
\cite[Definition 2.1]{MM}
\label{xxdef1.8}
Let $\Bbbk$ be a field.
A locally finite graded algebra $A$ with $A_0$ being
semisimple is called {\it Artin-Schelter Gorenstein}
(or {\it AS Gorenstein}) if the following conditions hold:
\begin{enumerate}
\item[(a)]
$A$ has finite injective dimension $d$ on both sides,
\item[(b)]
$\uExt^i_A(A_0, A) = \uExt^i_{A^{op}} (A_0, A) = 0$ for all
$i\neq d$ where $A_0 = A/A_{\geq 1}$, and
\item[(c)]
$\uExt^d_A(A_0, A)\cong A_0(l)$
and $\uExt^d_{A^{op}} (A_0, A)\cong
A_0(l)$ for some integer $l$.
\end{enumerate}
If moreover
\begin{enumerate}
\item[(d)]
$A$ has finite global dimension, and
\item[(e)]
$A$ has finite GK-dimension,
\end{enumerate}
then $A$ is called {\it Artin-Schelter regular} (or {\it AS regular}).
\end{definition}

In \cite[Definition 2.1]{MM}, the AS-Gorensteinness only 
requires that $A_0$ is finite dimensional and 
$\uExt_A^d(A_0,A)\cong A_0^*(l)$ where $A_0^*=\Hom_\Bbbk(A_0,\Bbbk)$. 
When $A_0$ is semisimple, the two definitions coincides since 
$A_0\cong A_0^*$ as $A_0$-bimodules in this case.

In the introduction, $R$ is connected graded AS regular.

When $A$ is noetherian, connected graded and PI of finite global
dimension, then $A$ is AS regular and CM \cite[Theorem 1.1]{SZ}.
It is conjectured that every noetherian connected graded AS regular
algebra is CM.

\begin{definition}
\label{xxdef1.9}
Let $R$ be a subalgebra of $B$ and let $\partial$ be a dimension
function defined on right $R$-modules and right $B$-modules.
\begin{enumerate}
\item[(1)]
We say $\partial$ is {\it weakly $B/R$-hereditary} if $\partial(M_R)
\leq \partial(M_B)$ for every f.g. right $B$-module $M$.
\item[(2)]
Suppose $B_R$ is f.g.. We say $\partial$ is
{\it $B/R$-hereditary} if for every f.g. right
$B$-module $M$, $\partial(M_R)=\partial(M_B)$.
\end{enumerate}
\end{definition}

\begin{lemma}
\label{xxlem1.10} Let $R$ be a subalgebra of $B$.
\begin{enumerate}
\item[(1)]
$\GKdim$ is always weakly $B/R$-hereditary.
\item[(2)]
Suppose that $R\subseteq B$ are noetherian and locally
finite graded algebras such that $B_R$ is f.g. Then $\GKdim$
is $B/R$-hereditary on graded right modules.
\end{enumerate}
\end{lemma}

\begin{proof} (1) This follows from Definition \ref{xxdef1.1}(2).

(2) Let $M$ be a f.g. graded right $B$-module. Since
$B_R$ is f.g., $M$ is a f.g. $R$-module.
By \eqref{E1.5.1}, $\GKdim M$ is only dependent on the
Hilbert series of $M$, not the module structure.
Therefore $\GKdim M_B=\GKdim M_R$.
\end{proof}

\section{A Morita type Theorem}
\label{xxsec2}

The main result of this section is Theorem \ref{xxthm2.4}.
We consider the following hypotheses to be used in different situations.

\begin{hypothesis}
\label{xxhyp2.1}
\begin{enumerate}
\item[(1)]
$A$ and $B$ are noetherian algebras.
\item[(2)]
Let $e$ be an idempotent in $B$ and $A=eBe$.
\item[(3)]
$\partial$ is a dimension function defined on right $A$-modules and
$B$-modules and $\partial(B)=:d\geq 2$.
\item[(4)]
$B$ is a $\partial$-CM$(2)$ algebra.
\item[(5)]
$eB$ and $Be$ are f.g. $A$-modules.
\item[(6)]
$\partial$ is exact both on right $A$-modules and on right $B$-modules.
\item[(7)]
For every f.g. right $B$-module $N$,
$\partial ((Ne)_A)(=\partial(N\otimes_B Be))\leq \partial (N_B)$.
\end{enumerate}
\end{hypothesis}

Note that Hypothesis \ref{xxhyp2.1}(7) is always satisfied if 
$\partial=\GKdim$ (cf. \cite[Proposition 5.6]{KL}).

Under Hypothesis \ref{xxhyp2.1}(2), $Be$ is a right
$A$-module and there is a natural algebra morphism
\begin{equation}
\label{E2.1.1}\tag{E2.1.1}
\varphi: \; B\to \End_{A}(Be),\quad \varphi(b)(b'e)=bb'e
\end{equation}
induced by the left multiplication. In this section
we will investigate when this morphism $\varphi$ is an
isomorphism. A special case of this situation is the Auslander theorem
\cite[Lemma 10.8]{Yo}.

We need use some weaker versions of CM condition. Let $\partial$
be a dimension function and $\alpha$ be an integer.
\begin{enumerate}
\item[(a)]
We say a noetherian algebra $B$ is {\it $\partial$-CM$^l(\alpha)$} 
if for any f.g. right $B$-module $M$,
$$\partial(M)\leq d-\alpha\Longrightarrow j(M)\geq \alpha$$
\item[(b)]
We say a noetherian algebra $B$ is {\it $\partial$-CM$^r(\alpha)$} 
if for any f.g. right $B$-module $M$,
$$\partial(M)\leq d-\alpha\Longleftarrow j(M)\geq \alpha.$$
\end{enumerate}
Then $B$ is $\partial$-CM$(\alpha)$ if and only if it is both
$\partial$-CM$^l(\alpha)$ and $\partial$-CM$^r(\alpha)$.

\begin{lemma}
\label{xxlem2.2}
Let $(A,B)$ satisfy Hypothesis {\rm{\ref{xxhyp2.1}(1,2,3)}}
and $\alpha$ be an integer such that $0<\alpha\leq d$. Suppose
that $B$ is $\partial$-CM$^l(\alpha)$. Let $N$
be a f.g. right $A$-module. If
$\partial(\Tor_i^A(N,eB))\leq d-\alpha$ for all $i\leq \alpha-1$, then
$\Ext_{A}^i(N,Be)=0$ for all $i\leq \alpha-1$.
\end{lemma}

\begin{proof} We have canonical isomorphisms induced by the adjointness
$$\Hom_{B}(N\otimes_A eB, B)\cong \Hom_{A}(N,\Hom_B(eB,B))
\cong \Hom_{A}(N,Be).$$
From these isomorphisms, we obtain a spectral sequence
\cite[Theorem 11.54]{Ro}
\begin{equation}
\label{E2.2.1}\tag{E2.2.1}
E_2^{pq}=\Ext^p_{B}(\Tor_q^A(N,eB),B)\Longrightarrow
\Ext^{p+q}_{A}(N,Be).
\end{equation}
By hypothesis, $B$ is $\partial$-CM$^l(\alpha)$, we have
$\Ext^p_{B}(\Tor_q^A(N,eB),B)=0$
for all $q\leq \alpha-1$ and $p\leq \alpha-1$ since
$\partial(\Tor_i^A(N,eB))\leq d-\alpha$
for all $i\leq \alpha-1$. By \eqref{E2.2.1}, we obtain
$\Ext^i_{A}(N,Be)=0$ for all $i\leq \alpha-1$.
\end{proof}

In the next lemma it is not necessary to assume $d\geq 2$.

\begin{lemma}
\label{xxlem2.3}
Let $(A,B)$ satisfy Hypothesis {\rm{\ref{xxhyp2.1}(1-3),(5-7)}}
and $\alpha$ be a positive integer. Suppose that
\begin{equation}
\label{E2.3.1}\tag{E2.3.1}
{\text{$\partial$ satisfies $\gamma_{d-\alpha,1}(eB)$.}}
\end{equation}
If $\partial (B/(BeB))\leq d-\alpha$, then
the functor $-\otimes_{\mathcal{B}}\mathcal{B}e: \;
\qmod_{d-\alpha} B\longrightarrow \qmod_{d-\alpha} A$ is an equivalence.
\end{lemma}

\begin{proof}
By hypothesis \eqref{E2.3.1}, $\partial$ satisfies
$\gamma_{d-\alpha,1}(eB)$. By Lemma \ref{xxlem1.3}, we have the
induced functor $-\otimes_\mathcal{A}e\mathcal{B}: \qmod_{d-\alpha} A
\longrightarrow\qmod_{d-\alpha} B$. Since $Be$ is projective as a left
$B$-module, the hypothesis of Lemma \ref{xxlem1.3} follows from
Hypothesis \ref{xxhyp2.1}(7) with $M=Be$ and $A$ and $B$ are switched.
By Lemma \ref{xxlem1.3}, we have the
induced functor $-\otimes_\mathcal{B} \mathcal{B}e: \qmod_{d-\alpha} B
\longrightarrow\qmod_{d-\alpha} A$. For any $N\in \mmod A$, we have
natural isomorphisms in $\qmod_{d-\alpha} A$:
$$\mathcal{N}\otimes_\mathcal{A}e\mathcal{B}\otimes_\mathcal{B}
\mathcal{B}e\cong \pi(N\otimes_A eB\otimes_B Be)\cong
\pi(N)=\mathcal{N}.$$
Taking $M\in\mmod B$,  we have natural isomorphisms in
$\qmod_{d-\alpha} B$:
$$\mathcal{M}\otimes_\mathcal{B}\mathcal{B}e\otimes_\mathcal{A}e
\mathcal{B}\cong \pi(M\otimes_B Be\otimes_A eB).$$
We claim that $\pi(M\otimes_B Be\otimes_A eB)=\pi(M)$.
Let $f: Be\otimes_A eB\to B$ be induced by the multiplication in $B$ and
factor $f$ as $f=hg$ with $h$ a monomorphism and $g$ an
epimorphism. Then we have two short exact sequences as follows,
\begin{equation}
\label{E2.3.2}\tag{E2.3.2}
0\to K\to Be\otimes_AeB\overset{g}\to BeB\to 0,
\end{equation}
\begin{equation}
\label{E2.3.3}\tag{E2.3.3}
0\to BeB\overset{h}\to B\to B/BeB\to 0.
\end{equation}
Applying the functor $-\otimes_BBe$ to \eqref{E2.3.2},
we obtain the exact sequence
$$0\to Ke\to Be\otimes_AeBe\to Be\to 0,$$
which results that $Ke=0$. Similarly, $eK=0$. Since $Be$ is a
f.g. right $A$-module, $Be\otimes_A eB$ is a f.g.
right $B$-module. Hence $K$ is a f.g. right $B/BeB$-module. 
Since $\partial (B/BeB)\leq d-\alpha$ and $\partial$ is exact on right $B$-modules, we have
$\partial (K)\leq d-\alpha$. Consider the exact sequence
$$M\otimes_BK\to M\otimes_B Be\otimes_AeB\to M\otimes_BBeB\to 0.$$
Since $M\otimes_B K$ is a f.g. $B/BeB$-module,
$\partial (M\otimes_B K)\leq d-\alpha$. Now we have
$\pi(M\otimes_B Be\otimes_AeB)\cong \pi(M\otimes_BBeB)$ in $\qmod_{d-\alpha} B$.
From the exact sequence (\ref{E2.3.3}), we obtain an exact sequence
$$\Tor_1^B(M,B/BeB)\to M\otimes_B BeB\to M\to M\otimes_BB/BeB\to0.$$
Note that both $\Tor_1^B(M,B/BeB)$ and $M\otimes_BB/BeB$ are f.g. right $B/BeB$-modules.
Since the $\partial$-dimension of $B/BeB$ is not larger than $d-\alpha$,
neither are the $\partial$-dimensions of $M\otimes_BB/BeB$ and
$\Tor_1^B(M,B/BeB)$. Hence we have natural isomorphisms in
$\qmod_{d-\alpha} B$: $\pi (M\otimes_B BeB)\cong \pi(M)$. Summarizing,
we obtain
$$\mathcal{M}\otimes_\mathcal{B}\mathcal{B}e\otimes_\mathcal{A}e
\mathcal{B}\cong \pi(M\otimes_B Be\otimes_A eB)\cong\pi(M)=\mathcal{M}$$
as desired.
\end{proof}

A special case of Lemma \ref{xxlem2.3} is the following well-known statement:
if $e$ is an idempotent in $B$ such that $BeB=B$, then $B$ is Morita
equivalent to $A:=eBe$. Most of Hypothesis \ref{xxhyp2.1} is not necessary in
this special case. Therefore Lemma \ref{xxlem2.3} can be viewed as a generalization
of Morita theorem. Specializing the last statement to a semisimple $H$-action
on an algebra $A$, if $e:=1\# \inth$ in $B:=A\# H$ is such that
$BeB=B$, then $A^H$ is Morita equivalent to $A\# H$, see
\cite[Corollary 4.5.4]{Mon}.

The main result of this section is the following Morita type theorem.

\begin{theorem}
\label{xxthm2.4}
Let $(A,B)$ satisfy Hypothesis {\rm{\ref{xxhyp2.1}(1-7)}}. Suppose
\begin{equation}
\label{E2.4.1}\tag{E2.4.1}
{\text{$\partial$ satisfies $\gamma_{d-2,1}(eB)$.}}
\end{equation}
Then the following statements are equivalent.
\begin{enumerate}
\item[(i)]
The functor $-\otimes_{\mathcal{B}}\mathcal{B}e: \;
\qmod_{d-2} B\longrightarrow \qmod_{d-2} A$ is an equivalence.
\item[(ii)]
The natural map $\varphi$ of \eqref{E2.1.1} is an isomorphism
of algebras.
\item[(iii)] $\partial (B/(BeB))\leq d-2$.
\item[(iv)]
$j(B/BeB)\geq 2$.
\end{enumerate}
\end{theorem}

\begin{proof} (i) $\Longrightarrow$ (ii).
Let $F=-\otimes_\mathcal{B}\mathcal{B}e$. Since $F$ is an
equivalence, we obtain isomorphisms of algebras:
\begin{equation}
\label{E2.4.2}\tag{E2.4.2}
\End_{\qmod_{d-2}B} \;\mathcal{B}\cong \End_{\qmod_{d-2}A}\;
F(\mathcal{B}) =\End_{\qmod_{d-2}A}\pi(Be).
\end{equation}
By assumption, $B$ is $\partial$-CM$^l(2)$. Hence, for $M\in
\GKmod_{d-2}B$,
$$\Hom_B(M,B)=0=\Ext_B^1(M,B).$$
As a consequence, $_BB$ does not contain any nonzero
submodule of $\partial$-dimension at most $d-2$.
Therefore, by definition,
$$\End_{\qmod_{d-2}B} \mathcal{B}=\underset{\longrightarrow}{\lim}\Hom_B(K,B),$$
where the limit runs over all the right submodule modules $K\subseteq B$ such
that $\partial (B/K)\leq d-2$. For such a submodule $K$, we have
$$\Hom_B(B/K,B)=0=\Ext_B^1(B/K,B).$$
Hence, a long exact sequence
implies that $\Hom_B(K,B)=\Hom_B(B,B)=B$. Thus, we have
$\End_{\qmod_{d-2}B} \mathcal{B}\cong B$.

Let $N\in \GKmod_{d-2} A$. By Lemma \ref{xxlem1.3}(1) and
\eqref{E2.4.1}, the hypothesis of Lemma \ref{xxlem2.2}
concerning $\partial(\Tor^A_i(N,eB))$ holds for $\alpha=2$.
Hence, by Lemma \ref{xxlem2.2},
\begin{equation}
\label{E2.4.3}\tag{E2.4.3}
\Hom_A(N,Be)=0=\Ext_A^1(N,Be)
\end{equation}
for all $N\in\GKmod_{d-2}A$. In particular, $Be$ does not
have any nonzero right $A$-submodule of $\partial$--dimension
$\leq d-2$. Using this and \eqref{E2.4.3}, we have
$$\Hom_{\qmod_{d-2}A}({\mathcal B}e,{\mathcal B}e)=
\underset{\longrightarrow}\lim\Hom_A(K,Be)\cong
\Hom_A(Be,Be),$$
where the limit runs over all the submodules
$K\subseteq Be$ such that $\partial (Be/K)\leq d-2$.
Combining with \eqref{E2.4.2}, one sees that
$B\cong \Hom_A(Be,Be)$, and the isomorphism is the natural
morphism given in statement (ii).

(ii) $\Longrightarrow$ (iv).
Let $f:Be\otimes_A eB\longrightarrow B$ be the $B$-bimodule
morphism defined by $f(be\otimes_A eb')=beb'$, which
was used in the proof of Lemma \ref{xxlem2.3}. Consider the
following composition of morphisms:
$$
B=\Hom_B(B,B)\overset{\Hom_B(f,B)}\longrightarrow
\Hom_B(Be\otimes_AeB,B)\qquad\qquad\qquad\qquad\qquad\qquad$$
$$\qquad\qquad\qquad\qquad\qquad\qquad   \; \overset{\phi}\longrightarrow
\Hom_A(Be,\Hom_B(eB,B))
\overset{\alpha}\longrightarrow\Hom_A(Be,Be),$$
\noindent
where $\phi$ is the natural isomorphism defined by the adjointness,
and $\alpha$ is the isomorphism induced by the isomorphism
$\Hom_B(eB,B)\cong Be$. One sees that the composition above is
equal to the natural morphism
$$\varphi:B\longrightarrow\text{\rm End}_A(Be), \quad
\varphi(b)(b'e)=bb'e.$$
By (ii), it is an isomorphism. Hence the morphism
$$\Hom_B(f,B):\Hom_B(B,B)\longrightarrow \Hom_B(Be\otimes_AeB,B)$$
is an isomorphism. Let us factor the morphism $f$ as follows:
$$\xymatrix{
Be\otimes_A eB \ar[rr]^{f} \ar[dr]_{g}
                &  &   B  \\
                & C, \ar[ur]^{h}
								}$$
where $C={\rm{im}} f=BeB$, $g$ is the epimorphism induced by
$f$ and $h$ is the inclusion map. Applying the functor
$\Hom_B(-,B)$ to the diagram above, we have
$\Hom_B(f,B)=\Hom_B(g,B)\circ\Hom_B(h,B)$. Since $g$ is an
epimorphism, it follows that $\Hom_B(g,B)$ is a monomorphism,
and hence it has to be an isomorphism, which in turn
implies that $\Hom_B(h,B)$ is an isomorphism. From the exact
sequence
$$0\longrightarrow C\overset{h}\longrightarrow
B\longrightarrow B/BeB\longrightarrow 0,$$
we obtain the following exact sequence
%
$$0\longrightarrow \Hom_B(B/BeB,B)\longrightarrow\Hom_B(B,B)
\qquad\qquad\qquad\qquad\qquad\qquad$$
$$\qquad\qquad\qquad\qquad\qquad\qquad
\overset{\Hom_B(h,B)}\longrightarrow\Hom_B(C,B)
\longrightarrow \Ext_B^1(B/BeB,B)\longrightarrow 0.$$

Since $\Hom_B(h,B)$ is an isomorphism, we obtain $\Hom_B(B/BeB,B)=0$
and $\Ext_B^1(B/BeB,B)=0$. Hence $j(B/(BeB))\ge2$.

(iv) $\Longrightarrow$ (iii) This follows from
$\partial$-CM$^r(2)$.

(iii) $\Longrightarrow$ (i). This is a special case of Lemma
\ref{xxlem2.3} when $\alpha=2$.
\end{proof}

\begin{remark}
\label{xxrem2.5} We have the following observations concerning
the proof of Theorem \ref{xxthm2.4}.
\begin{enumerate}
\item[(1)]
In the proof of (i) $\Longrightarrow$ (ii), only $\partial$-CM$^l(2)$
is used (not the entire $\partial$-CM$(2)$).
\item[(2)]
In the proof of (iv) $\Longrightarrow$ (iii), only $\partial$-CM$^r(2)$
is used (not the entire $\partial$-CM$(2)$). In fact, we only used
the fact that $j(B/BeB)\geq 2$ implies that
$\partial(B/BeB)\leq d-2$.
\item[(3)]
In the proof of (iii) $\Longrightarrow$ (i), $\partial$-CM$(2)$
is not needed.
\item[(4)]
(iii) $\Longrightarrow$ (iv) Follows from the definition of
$\partial$-CM$(2)$.
\end{enumerate}
\end{remark}

Now let $B=\bigoplus_{i\ge0}B_i$ be a noetherian locally finite graded
algebra. Let $e\in B_0$ be an idempotent. Then $A:=eBe$ is a noetherian
locally finite graded algebra. Following Definition \ref{xxdef1.4},
$B$ is said to be {\it graded CM} if
$\GKdim B=d<\infty$ and $j(M)+\GKdim K=\GKdim B$ for all
f.g. graded right $B$-modules $M$.

Note that the GK-dimension is always exact on modules over noetherian
locally finite graded algebras [Remark \ref{xxrem1.5}(2)].
By Lemma \ref{xxlem1.6}, $\GKdim$ satisfies $\gamma_{n,i}$.
Note that, by \eqref{E1.5.1}, $\GKdim$ satisfies Hypothesis
\ref{xxhyp2.1}(7) in the locally finite graded case. Therefore we
obtain the following graded version of Theorem \ref{xxthm2.4}.

\begin{theorem}
\label{xxthm2.6} Let $B=\bigoplus_{i\ge0}B_i$ be a noetherian
locally finite graded algebra. Suppose $(A,B)$ satisfies
Hypothesis {\rm{\ref{xxhyp2.1}(1-5)}} in the graded setting
with $e\in B_0$ and $\partial=\GKdim$.
Then the following statements are equivalent.
\begin{enumerate}
\item[(i)]
The functor $-\otimes_{\mathcal{B}}\mathcal{B}e:
\qgr_{d-2} B\longrightarrow \qgr_{d-2} A$ is an equivalence.
\item[(ii)]
The natural map $\varphi:B\longrightarrow\text{\rm End}_A(Be)$
defined by $\varphi(b)(b'e)=bb'e$ is an isomorphism of graded algebras.
\item[(iii)]
$\GKdim B/(BeB)\leq d-2$.
\item[(iv)]
$j(B/BeB)\geq 2$.
\end{enumerate}
\end{theorem}

\section{Hopf algebra actions on graded algebras}
\label{xxsec3}

In this section  we consider locally finite graded algebras and
let $\partial$ be a dimension function on locally finite graded
modules. Starting from Proposition \ref{xxpro3.3}, we further assume
that $\Bbbk$ is a field.

Let $(H,\Delta,\varepsilon, S)$ be a Hopf algebra that is free
of finite rank over $\Bbbk$.
Let $R$ be a noetherian locally finite graded
algebra. Assume that $R$ is a graded left $H$-module algebra, i.e., $H$
acts on $R$ homogeneously. Then the smash product $R\#H$ is a noetherian
locally finite graded algebra.

Suppose that $H$ has a (left and right) integral $\inth$ such that
$\varepsilon(\inth)=1$. If $\Bbbk$ is a field, this is equivalent to
the fact that $H$ is semisimple. Let $e=1\# \inth\in R\#H$. One sees
that $e$ is an idempotent of $R\#H$. Write the fixed subring of the
$H$-action on $R$ as
\begin{equation}
\label{E3.0.1}\tag{E3.0.1}
R^H=\{r\in R\mid h\cdot r=\varepsilon(h)r,\forall\ h\in H\}.
\end{equation}
Then $R^H$ is a graded subalgebra of $R$. The following is well-known \cite{Mon}.

\begin{lemma}
\label{xxlem3.1} With the notation as above. The following hold.
\begin{enumerate}
\item [(1)]
$R^H$ is noetherian.
\item [(2)]
$R$ is f.g. as a left graded $R^H$-module and
as a right graded $R^H$-module.
\item [(3)]
The map $R^H\longrightarrow e(R\#H)e$, sending $r\mapsto e(r\# 1)e$,
is an isomorphism of graded algebras.
\item [(4)]
The map $R\longrightarrow(R\#H)e$, sending $r\mapsto (r\#1)e$, is
an isomorphism of graded $(R,R^H)$-modules.
\item [(5)]
The map $R\longrightarrow e(R\#H)$, sending $r\mapsto e(r\#1)$, is
an isomorphism of graded $(R^H,R)$-modules.
\end{enumerate}
\end{lemma}

\begin{proof} (1) This follows by the proof of
\cite[Corollary 4.3.5]{Mon}.

(2) This follows by the proof of
\cite[Theorem 4.4.2]{Mon}.

(3,4,5) By direct computation.
\end{proof}

\begin{hypothesis}
\label{xxhyp3.2}
\begin{enumerate}
\item[(1)]
$R$ is a noetherian, locally finite, graded algebra.
\item[(2)]
$H$ is a Hopf algebra acting on
$R$ homogeneously, and $H$ is free of finite rank over $\Bbbk$ with
integral $\inth$ such that $\epsilon(\inth)=1$.
\item[(3)]
Let $B$ be the algebra $R\# H$ with $e:=1\#\inth\in B$.
Identifying $R^H$ with $eBe$ by Lemma {\rm{\ref{xxlem3.1}(3)}},
$R$ with $Be$ by Lemma {\rm{\ref{xxlem3.1}(5)}}.
\item[(4)]
Let $\partial$ be an exact dimension function on f.g. right graded $B$-modules,
$R$-modules, and $R^H$-modules.
This is automatic if $\partial=\GKdim$.
\item[(5)]
$\partial$ is $B/R$-hereditary. This is automatic
if $\partial=\GKdim$ by Lemma {\rm{\ref{xxlem1.10}(2)}}.
\item[(6)]
$\partial(R)=d\geq 2$ and $R$ is graded $\partial$-CM$(2)$
\end{enumerate}
\end{hypothesis}

Here is our first result in this section.
For the rest of this section we assume that $\Bbbk$ is a field.

\begin{proposition}
\label{xxpro3.3}
Retain Hypothesis {\rm{\ref{xxhyp3.2}(1-5)}}.
If $R$ is $\partial$-CM {\rm{(}}respectively,
$\partial$-CM$(s)$, $\partial$-CM$^l(s)$,
$\partial$-CM$^r(s)${\rm{)}}, then so is
$B:=R\#H$.
\end{proposition}

\begin{proof}
Since $B$ is isomorphic to $R^{\oplus \dim H}$ as a right
$R$-module, we have
$$\partial(B_B)=\partial(B_R)=\partial(R_R)=:d.$$

Let $M$ be a f.g. graded right $B$-module with $\partial$-dimension
$i$. Since $\partial$ is $B/R$-hereditary, $M$, viewed as a
right $R$-module, has $\partial$-dimension $i$, or
$\partial(M_B)=\partial(M_R)$. It remains to
show that $j(M_B)=j(M_R)$.

The cohomology $\uExt_B^j(M,B)$ can be computed by $\uExt^j_R(M,B)$. Since
$H$ is semisimple, by \cite[Corollary 7.6]{Ne}, we have
\begin{equation}\label{E3.3.1}\tag{E3.3.1}
\uExt_B^j(M,B)\cong \uExt_R^j(M,B)^H.
\end{equation}
Equation \eqref{E3.3.1} follows also from a more general result
\cite[Theorem 3.3]{St} together with \cite[Lemma 9.1.9]{We}.

We recall the right $H$-module structure on $\uExt_R^j(M,B)$ as follows.
Let
$$\cdots\longrightarrow P^{-n}\longrightarrow\cdots\longrightarrow
P^0\longrightarrow M_B\longrightarrow 0$$
be a f.g. projective resolution of $M_B$.
Let $N$ be a right $B$-module. For $j\ge0$, $\uHom_\Bbbk(P^{-j},N)$ has
a right $H$-module structure defined by
$$(f\cdot h)(p)=f(p(Sh_{(1)}))h_{(2)}$$
for all $p\in P^{-j}$, $f\in \uHom_\Bbbk(P^{-j},N)$ and $h\in H$,
where $S$ is the antipode of $H$ and we use Sweedler's notation
$\Delta(h)=h_{(1)}\otimes h_{(2)}$. One may check that $\uHom_R(P^{-j},N)$ is an $H$-submodule of $\uHom_\Bbbk(P^{-j},N)$. The right $H$-action is
compatible with the differential of the complex $\uHom_R(P^\bullet,N)$,
from which we obtain the right $H$-action on $\uExt_R^j(M,N)$.

We may form another smash product $H\# R$ as follows \cite[Lemma 2.1]{RRZ}:
as a vector space $H\# R=H\otimes R$, the product is defined by
$$(h\#r)(g\#r')=hg_{(2)}\#(S^{-1}g_{(1)}\cdot r)r'.$$
The algebras $R\#H$ and $H\#R$ are isomorphic through the isomorphism:
$$\varphi:R\#H\longrightarrow H\#R,$$
defined by
$$r\#h\mapsto h_{(2)}\#S^{-1}h_{(1)}\cdot r.$$
For all $j\ge0$, we have the following morphism
$$\Psi:\uHom_R(P^{-j},R\#H)\longrightarrow\uHom_R(P^{-j},H\# R),$$
defined by
$$\Psi(f)=\varphi\circ f.$$
We also may define a morphism
$$\Phi:\uHom_R(P^{-j},R)\otimes H\longrightarrow \uHom_\Bbbk(P^{-j},H\# R),$$
defined by
$$\Phi(f\otimes h)(p)=h\#f(p).$$ A straightforward check shows that $\Phi$ takes value in $\uHom_R(P^{-j},H\# R)$. Hence, we indeed obtain a morphism $\Phi:\uHom_R(P^{-j},R)\otimes H\longrightarrow \uHom_R(P^{-j},H\# R)$.
One may check that both $\Psi$ and $\Phi$ are isomorphisms of
vector spaces and are compatible with the differentials.
Through the isomorphism $\varphi$, we view $H\#R$ as a right $R\#H$-module.
Note that $R$ can be viewed as a right $R\#H$-module via the action
$$r\cdot(r'\#h)=S^{-1}h(rr')$$ for $r,r'\in R$ and $h\in H$. Hence
both  $\uHom_R(P^{-j},H\#R)$ and $\uHom_R(P^{-j},R)$ are
right $H$-modules. Now $\uHom_R(P^{-j},R)\otimes H$ is also a
right $H$-module by the usual $H$-action defined by
$$(f\otimes g)\cdot h=f\cdot h_{(1)}\otimes g h_{(2)}$$
for $f\in \uHom_R(P^{-j},R)$ and $g,h\in H$.
Since $\varphi$ is an isomorphism of algebras,
$\Psi$ is an isomorphism of right $H$-modules.
We next check that $\Phi$ is a right $H$-module
morphism, and hence it is indeed an isomorphism
of right $H$-modules. For $f\in \uHom_R(P^{-j},R)$,
$h,g\in H$ and $p\in P^{-j}$, we have
$$\begin{array}{ccl}
    [\Phi(f\otimes h)\cdot g](p) & = & [\Phi(f\otimes h)(pSg_{(1)})]g_{(2)} \\
     & = & [h\# f(pSg_{(1)})]g_{(2)} \\
     & = & hg_{(3)}\#S^{-1}g_{(2)}f(p(Sg_{(1)})),
  \end{array}
$$
$$
\begin{array}{ccl}
  \Phi((f\otimes h)g)(p) & = & \Phi(f\cdot g_{(1)}\otimes h g_{(2)})(p) \\
   & = & hg_{(2)}\# (f\cdot g_{(1)}) (p) \\
   & = & hg_{(3)}\# f(pSg_{(1)})\cdot g_{(2)}\\
   & = & hg_{(3)}\# S^{-1}g_{(2)}f(pSg_{(1)}).
\end{array}
$$ Hence $\Phi$ is a right $H$-module morphism.
Therefore the composition $\Phi^{-1}\Psi$
induces an isomorphism of complexes of right $H$-modules
$$\uHom_R(P^{\bullet},R\#H)\longrightarrow \uHom_R(P^{\bullet},R)\otimes H.$$
Applying $(-)^H$, we obtain an isomorphism of complexes
$$\uHom_R(P^{\bullet},R\#H)^H\cong (\uHom_R(P^{\bullet},R)\otimes H)^H.$$
Since $(-)^H$ is an exact functor,
taking cohomology $H^j(- )$ and taking the
invariants $(- )^H$ commutes. Therefore we have
$$\uExt_R^j(M,R\#H)^H\cong (\uExt^j_R(M,R)\otimes H)^H.$$
Combining with the isomorphism \eqref{E3.3.1}, we finally obtain
\begin{equation}
\label{E3.3.2}\tag{E3.3.2}
\uExt_B^j(M,B)\cong (\uExt^j_R(M,R)\otimes H)^H.
\end{equation}

By \eqref{E3.3.2}, it is clear that $j(M_B)\geq j(M_R)$.
Assume that $x_1,\dots,x_n$ is a $\Bbbk$-basis of $H$. Then we
may write the coproduct of $\inth$ as $\Delta(\inth)=\sum_{s=1}^ny_s\otimes x_s$. Since the one-dimensional vector space span$_\Bbbk \{\int\}$ is an ideal of $H$,
we see that $y_1,\dots,y_n$ generate $H$ by \cite[Proposition 2.6]{Sw}.
If $\uExt^{j}_R(M,R)\neq0$, there is
a nonzero element $\alpha\in\uExt^{j}_R(M,R)$ such that
$\alpha y_s\neq0$ for some $s$. Now
$\sum_{s=1}^n\alpha y_s\otimes x_s\neq 0$, which is in
$(\uExt^{j}_R(M,R)\otimes H)^H$. Therefore
$\uExt_B^{j}(M,B)\neq0$. This means that $j(M_B)\leq
j(M_R)$. Therefore $j(M_B)=j(M_R)$.
\end{proof}

Combining Lemmas \ref{xxlem2.2}, \ref{xxlem1.6} and \ref{xxlem3.1}
and Proposition \ref{xxpro3.3}, we have the following consequences.

\begin{proposition}
\label{xxpro3.4}
Retain Hypothesis {\rm{\ref{xxhyp3.2}(1,2,3)}} with
$\partial=\GKdim$ and assume that $\Bbbk$ is a field and that
$R$ is CM$(\alpha)$ for
a positive integer $\alpha$. Then
$\uExt_{R^H}^i(N,R)=0$ for all $N\in \GKgr_{d-\alpha}R^H$ and all
$i\leq \alpha-1$.
\end{proposition}

\begin{proof} Let $(A,B)=(R^H,R\# H)$. Then Hypothesis
\ref{xxhyp2.1}(1,2,3) can be verified easily. By
Proposition \ref{xxpro3.3}, $B$ is CM$(\alpha)$. The
assertion follows from Lemmas
\ref{xxlem1.6} and \ref{xxlem2.2}.
\end{proof}

Assisted by Proposition \ref{xxpro3.3}, Theorem \ref{xxthm2.6} now
reads in the following form. Recall that the pertinency of $H$-action
on $R$ [Definition \ref{xxdef0.1}] is defined to be
$$\Pty(R,H)=\GKdim R-\GKdim (R\# H)/(e)$$
where $e=1\#\inth$. 

\begin{theorem}
\label{xxthm3.5}
Retain Hypothesis {\rm{\ref{xxhyp3.2}(1-3)}} with $\partial=\GKdim$.
Assume that $R$ is CM$(2)$ with $\GKdim R=d\geq 2$.
Set $(A,B)=(R^H,R\#H)$. Then the following are equivalent.
\begin{enumerate}
\item [(i)]
The functor $-\otimes_{\mathcal{B}}\mathcal{R}:\qgr_{d-2} B
\longrightarrow \qgr_{d-2} A$ is an equivalence.
\item[(ii)]
The natural map
$\varphi:B\longrightarrow\text{\rm End}_A(R)$ is an
isomorphism of algebras.
\item[(iii)]
$\Pty(R,H)\geq 2$.
\item[(iv)]
$H$-action is  homologically small in the sense of
Definition {\rm{\ref{xxdef0.2}(2)}}.
\end{enumerate}
\end{theorem}

\begin{proof} This is an immediate consequence of Theorem
\ref{xxthm2.6} after matching up
with notations and verifying the hypotheses by
Proposition \ref{xxpro3.3}. Note that (iii) and (iv) are
equivalent by CM$(2)$.
\end{proof}

\begin{remark}
\label{xxrem3.6}
(1) Theorem \ref{xxthm0.3} is a special case of Theorem \ref{xxthm3.5}.

(2) A version of Theorem \ref{xxthm3.5} holds when $R$ is ungraded
(or $R$ is not locally finite) as long as
\begin{enumerate}
\item[(i)]
the GK-dimension on right $R^H$-modules are exact,
\item[(ii)]
the GK-dimension is $R\# H/R^H$-hereditary,
\item[(iii)]
$\gamma_{n,1}(eB)$ holds for $\GKdim$.
\end{enumerate}

(3)
If $R={\mathbb C}[x_1,\dots,x_d]$ is the polynomial algebra, and $G$
a finite small subgroup of $\text{\rm GL}_d({\mathbb C})$, the
Auslander theorem says the natural map
$R\#\Bbbk G\longrightarrow \text{\rm End}_{R^G}(R)$ as in
Theorem \ref{xxthm3.5} is an isomorphism of graded algebras
\cite{Au1,Yo,IT}. By Theorem \ref{xxthm3.5}, $\Pty(R,G)\geq 2$.
See comments in \cite[Section 7]{BHZ}.

(4) If $R/R^H$ is an $H^*$-dense Galois extension in the sense of 
\cite{HVZ2}, then $\Pty(R,H)=d$. In this case, 
\cite[Theorems 2.4 and 3.8]{HVZ2} implies that the statement (ii) 
of Theorem \ref{xxthm3.5} holds. Hence the equivalence of (ii) and 
(iii) may be viewed as a generalization of \cite[Theorem 3.8(ii)]{HVZ2}. 
In particular, if $\Pty(R,H)=\GKdim R=2$, then the equivalence of (i) 
and (ii) in Theorem \ref{xxthm3.5} is indeed a part of \cite[Theorem 3.8]{HVZ2}. 
On the other hand, \cite[Theorem 3.8]{HVZ2} doesn't assume that the algebra $R$ is CM; 
instead, there is an additional assumption on the depth of $R$ as a right $A$-module.

(5) If $G$ is a finite group of graded automorphisms of $R$, and $R$ is 
AS-regular and CM, then Theorem \ref{xxthm3.5} implies \cite[Theorem 3.7]{MU} 
since the ampleness of the group action implies that $\Pty(R,H)=\GKdim R$. 
However, the isomorphism in \cite[Theorem 3.7]{MU} holds for more general 
AS-regular algebras.
\end{remark}

Next we prove Theorem \ref{xxthm0.6}.

\begin{theorem}
\label{xxthm3.7}
Retain Hypothesis {\rm{\ref{xxhyp3.2}(1,2,3)}}.
Let $\alpha$ be an integer no more than $\Pty(R,H)$, and let $d=\GKdim R$.
Then there is an equivalence of categories
$$\qgr_{d-\alpha} R^H\cong \qgr_{d-\alpha} R\# H.$$
\end{theorem}

\begin{proof} By Lemma \ref{xxlem1.6}, hypothesis
\eqref{E2.3.1} holds for $\partial=\GKdim$ on graded modules.
Hypotheses \ref{xxhyp2.1}(1-3)(5-7) can be checked one by one.
Therefore the assertion follows by Lemma \ref{xxlem2.3}.
\end{proof}

\begin{corollary}
\label{xxcor3.8}
Retain Hypothesis {\rm{\ref{xxhyp3.2}(1,2,3)}}.
Let $R$ be an AS regular algebra.
If $\Pty(R,H)=\GKdim R$, then $R^H$ has graded
isolated singularities.
\end{corollary}

\begin{proof} Assume $d=\GKdim R$. By Theorem \ref{xxthm3.7},
$\qgr_{d-\alpha} R^H\cong \qgr_{d-\alpha} R\# H$
where the latter has finite global
dimension. Taking $\alpha=d$, the assertion follows by the
definition, see \cite[Definition 2.2]{Ue}.
\end{proof}

For the rest of this section we make some observations the condition
to ensure that $\Pty(R,H)\geq 1$. We recall a definition below.

\begin{definition}
\label{xxdef3.9}
A left $H$-module $M$ is called
{\it inner faithful} if $IM\neq 0$ for every nonzero Hopf ideal $I$
of $H$. We say that an $H$-action on $A$ is inner faithful, if $A$
is an inner faithful left $H$-module.
\end{definition}

\begin{lemma}
\label{xxlem3.10}
Retain Hypothesis {\rm{\ref{xxhyp3.2}(1-5)}} with $\partial=\GKdim$.
Assume that $\Bbbk$ is a field.
Suppose that $R$ is a domain and $(A, B)=(R^H, R\# H)$.
\begin{enumerate}
\item[(1)]
\cite[Theorem 0.5(i)]{SkV}
$B$ is semiprime.
\item[(2)]
\cite[Corollary 3.4]{BCF}
$B$ is prime if and only if the canonical algebra homomorphism
$\phi: B\to \End_{R^H}(R)$ is injective.
\item[(3)]
If the map $\phi$ in part {\rm{(2)}} is injective, then the
$H$-action on $R$ is inner faithful. As a consequence,
if $B$ is prime, then the $H$-action is inner faithful.
\item[(4)]
$B$ is prime if and only if $\Pty(R,H)\geq 1$.
\end{enumerate}
\end{lemma}

\begin{proof} (1) This is a special case of \cite[Theorem 0.5(i)]{SkV}.

(2) In order to apply \cite[Corollary 3.4]{BCF} we need to verify 
the hypothesis in \cite[Corollary 3.4]{BCF}, namely, the condition 
in \cite[Lemma 3.3(2)]{BCF}. Since $R$ is a domain, 
so is $A:=R^H$. Hence $A$ is prime. Let $I$ be a nonzero one-sided 
ideal of $R$ which is stable under the $H$-action. Since $R$ is a 
domain, $\GKdim R/I<\GKdim R$. If $A\cap I=0$, then 
$$\GKdim R/I=\GKdim A=\GKdim R,$$ 
a contradiction. Hence $A\cap I\neq 0$.
So, Condition \cite[Lemma 3.3(2)]{BCF} is satisfied. The assertion 
now follows from \cite[Corollary 3.4]{BCF}. 

(3) If the $H$-action on $R$ is not inner faithful, there is
a Hopf ideal $0\neq J\subseteq H$ such that $J R=0$. Thus
$\phi(1\# j)=0$ for all $j\in J$. The assertion follows. The consequence
follows from part (2).

(4) Let $d=\GKdim R$. First we suppose that $\Pty(R,H)\geq 1$.
By part (1), $B$ is semiprime. Since $B$ is a free module
over $R$ and $R$ is a domain, $B$ is $d$-pure
for $R$-modules, and hence for $B$-modules with respect to the GK-dimension
in the sense of \cite[Definition 0.3]{ASZ2}. By Lemma \ref{xxlem1.6},
Hypothesis \eqref{E2.3.1} holds for any $\alpha$, and hence for $\alpha=1$.
By Lemma \ref{xxlem2.3}, $\qmod_{d-1} B\cong \qmod_{d-1} A$. Since $A$
is an Ore domain, $\qmod_{d-1} A$ is equivalent to a module category over
the quotient division ring $Q(A)$ of $A$. Therefore, $\qmod_{d-1} B$ is
equivalent to a module category over a division ring.
Since $B$ is $\GKdim$-pure, $\qmod_{d-1} B\cong {\text{mod}}\; Q(B)$ where $Q(B)$ is
the semisimple artinian quotient ring of $B$. This implies that $Q(B)$ is
simple and $B$ is prime.

Conversely, since $B$ is prime and $I=BeB\subseteq B$ is a nonzero ideal, 
$\GKdim B/I\leq d-1$. Hence $\Pty(R,H)\ge1$.
\end{proof}

Lemma \ref{xxlem3.10} were suggested by
Ellen Kirkman. The authors thank her for sharing her ideas.

\section{BGG correspondence}
\label{xxsec4}

From now on let $\Bbbk$ be a field.
In this section, we study the quotient category $\qgr_{d-\alpha}R^H$
as appeared in the last section. If $R$ is noetherian, connected graded,
AS regular and if the homological determinant of the $H$-action on $R$
is trivial, then $R^H$ is AS Gorenstein \cite[Theorem 0.1]{KKZ1}. Under 
the setting of Theorem \ref{xxthm3.5} (assuming $\Pty(R,H)\geq 2$), 
the quotient category $\qgr_{d-2}R^H$ is equivalent to the quotient 
category $\qgr_{d-2}R\#H$. Since $R\# H$ has finite global dimension,
while $R^H$ has infinite global dimension when $H$ is non-trivial 
\cite[Theorem 0.6]{CKWZ1}, it is relatively easier to understand 
homological properties of the category $\qgr_{d-2}R\#H$ than to do the
category $\qgr_{d-2}R^H$ directly. 

We are now working over Koszul AS regular algebras. Let
$B=\bigoplus_{n\ge0}B_n$ be a noetherian locally finite graded
algebra such that $B_0$ is a semisimple algebra. Recall that a
right graded $B$-module is called a {\it Koszul module} \cite{P,BGS}
if $M$ has a linear graded projective resolution:
$$\cdots \longrightarrow P^{-n}\longrightarrow\cdots\longrightarrow
P^{-1}\longrightarrow P^0\longrightarrow M\longrightarrow0,$$
where $P^{-n}$ are graded projective modules generated in degree $n$
for all $n\ge0$. If the right graded $B$-module $B_0:=B/B_{\geq 1}$
is a Koszul module, then $B$ is called a {\it Koszul algebra}.

Let $B$ be a graded algebra and let
\begin{equation}
\label{E4.0.1}\tag{E4.0.1}
E=E(B):=\oplus_{n\ge0}\uExt^n_B(B_0,B_0)
\end{equation}
be the Yoneda algebra of $B$.
By \cite[Proposition 2.9.1 and Theorem 2.10.1]{BGS}, if $B$ is a
Koszul algebra, then the Yoneda algebra $E$ is also Koszul.

Let $B$ be a locally finite graded algebra.
Let $\{S_1,\cdots, S_w\}$ be the complete
set of non-isomorphic graded right $B$-module concentrated in
degree 0, and let $Q_i$ be the projective
cover of $S_i$, for $i=1,\cdots, w$. Define
\begin{equation}
\label{E4.0.2}\tag{E4.0.2}
B'=\uHom_{B} (\bigoplus_{i=1}^w Q_i,\bigoplus_{i=1}^w Q_i).
\end{equation}
In applications we also assume that the degree zero piece $B_0$
of $B$ is semisimple. The following lemma is well-known.

\begin{lemma}
\label{xxlem4.1} Retain the notations as above and assume that
$B_0$ is semisimple.
\begin{enumerate}
\item[(1)]
$B'$ is graded Morita equivalent to $B$. Namely,
there is an equivalence of graded module categories
$$\Gr B'\cong \Gr B,$$
which induces an equivalence of triangulated categories
$$D(\Gr B')\cong D(\Gr B).$$
\item[(2)]
$E(B)$ is graded Morita equivalent to $E(B')$.
\item[(3)]
$B$ is noetherian {\rm{(}}respectively, AS Gorenstein, AS regular,
Koszul{\rm{)}} if and only if so is $B'$.
\item[(4)]
If $\Bbbk$ is algebraically closed, then $B'_0$ is a
finite direct sum of $\Bbbk$.
\end{enumerate}
\end{lemma}

For the rest of this section, we assume that {\it $\Bbbk$ is
algebraic closed}. Then $B'_0$ is a
finite direct sum of $\Bbbk$ [Lemma \ref{xxlem4.1}(4)].
In this case $B'$ is called the {\it basic algebra} of $B$.
Assume that $R$ is as in Theorem \ref{xxthm3.5} and that $R$ is a
connected graded Koszul AS regular algebra. In view of \cite[Theorem 3.2]{HVZ1}, $R\#H$ is a Koszul algebra with finite dimensional Koszul dual. Hence $R\#H$ has finite global dimension. By \eqref{E3.3.2} and its dual version, it follows that $R\#H$ is an AS-regular algebra.
Since the graded Morita equivalence of algebras preserves almost
all algebraic invariants [Lemma \ref{xxlem4.1}], we will
study the basic algebra of $R\#H$ instead. In general, one can
replace $B$ by its basic algebra $B'$ if necessary.

From now on, $B=B_0\oplus B_1\oplus \cdots$ is a noetherian locally
finite graded algebra such that $B_0$ is a finite direct sum of
$\Bbbk$. The following lemma is needed.

\begin{lemma}
\label{xxlem4.2}
\cite[Theorem 5.1]{Ma}
Let $B$ be a Koszul algebra. Then $B$ is AS regular if and
only if $E:=\bigoplus_{n\ge0}\uExt^n_B(B_0,B_0)$ is a graded
Frobenius algebra.
\end{lemma}

Let $B$ be a Koszul algebra. We may construct a
differential bigraded algebra $\widetilde{B}$ by setting
$\widetilde{B}^0_j=B_j$ and $\widetilde{B}^i_j=0$ for all $i\neq0$
with zero differential.
We may construct a differential bigraded algebra $\widetilde{E}$
associated to $E$ \eqref{E4.0.1}
by setting $\widetilde{E}^i_{-i}=E_i$ and $\widetilde{E}^i_j=0$ for
all $i\neq -j$ with zero differential. A differential bigraded
$\widetilde{B}$-module $M$ is a bigraded $\widetilde{B}$-module
$M=\oplus_{i,j\in \mathbb Z} M^i_j$ together with a differential
$d$ such that $d(M^i_j)\subseteq M^{i+1}_j$ for all $i,j\in \mathbb Z$.
Note that the derived category of differential bigraded right
$\widetilde{B}$-modules (respectively, $\widetilde{E}$-modules)
is equivalent to the derived category of right graded $B$-modules
(respectively, $E$-modules). Then the general
Koszul duality of bigraded differential graded algebras induces
following duality between the derived category of $B$ and of $E$
\cite[Remark 4.1]{HW}, which is a slightly different from the
equivalence obtained in \cite{BGS}. Note that Smith-Van den Bergh
stated a similar result in \cite{SmV}.

\begin{proposition}
\label{xxpro4.3} \cite[Section 2.4]{SmV}
Let $B$ be a Koszul AS regular algebra. We have a duality of
triangulated categories:
$$K:D^b(\gr B)\longrightarrow D^b(\gr E),$$
where $K(M)=G(\RHom_B(M,B_0))$ for $M\in D^b(\gr B)$, and $G$
is the regrading functor defined by $G(X)^n_s=X^{n+s}_{-s}$ for
any complex of graded modules $X^\cdot$.
\end{proposition}

Note that the functor $K$ commutes with the shift functors in
the following way
$$K(X^\cdot[1])=K(X^\cdot)[-1] \text{ and } K(X^\cdot(1))=
K(X^\cdot)[-1](1),$$
where as usual $[1]$ is the shift of the complexes and $(1)$ is
the shift of the gradings.

For $M\in \gr E$, let $\Omega(M)$ be the first syzygy of the
minimal graded projective resolution of $M$. Since $E$ is a
graded Frobenius algebra, $\Omega$ is an auto-equivalence of
the stable category $\underline{\gr} E$. Moreover,
$\underline{\gr} E$ is a triangulated category with the shift
functor $\Omega^{-1}$. We next show that the triangulated
category $\underline{\gr}E$ admits some thick triangulated
subcategories. Some terminologies are needed next. Let $M$ be a
finite dimensional right graded $E$-module. Take a minimal
graded projective resolution of $M$ as follows
$$\cdots\longrightarrow P^{-n}\longrightarrow\cdots\longrightarrow
P^0\longrightarrow M\longrightarrow 0.$$
Recall that the {\it complexity}  of $M$, denoted by $c(M)$,  is
the least integer $\lambda\ge0$ such that $\dim P^{-n}<a\, n^{\lambda-1}$
for almost all $n$, where $a>0$ is a fixed number. The definition
of the complexity can be modified even for noninteger $\lambda$. Note
that $c(M)=0$ if and only if $M$ is a finite dimensional graded
projective $E$-module (when $E$ is Frobenius). The complexity
of a module over a finite dimension algebra is an important
invariant in studying the representations of the algebra
\cite{Ca,CDW,GLW}.

Given a number $\lambda\geq 0$, let $\mathcal{C}_\lambda$
be the full subcategory of $\underline{\gr} E$ consisting of objects $M$
with $c(M)\leq \lambda$. Similar to \cite[Proposition 1.3]{CDW},
we have the following lemma, whose proof is almost the same as
that of \cite[Proposition 1.3]{CDW}.

\begin{lemma}
\label{xxlem4.4}
The induced category $\mathcal{C}_\lambda$ is a thick triangulated
subcategory of $\underline{\gr} E$.
\end{lemma}

\begin{proof}
Let $0\to K\to M\to N\to0$ be an exact sequence in $\gr E$.
Then one sees $c(M)\leq\max\{c(K),c(N)\}$. Hence $M\in
\mathcal{C}_\lambda$ if $K$ and $N$ are in $\mathcal{C}_\lambda$.
Moreover, it follows from \cite[Lemma 2.5]{AR} or \cite[Lemma 2.7]{GLW}
that if $K$ and $M$ (resp. $M$ and $N$) are in $\mathcal{C}_\lambda$,
then so is $N$ (resp. $K$). Note that the projective modules and injective module coincide in the category gr$E$ and any projective module has zero complexity. If $U\to V\to Q\to \Omega^{-1}(U)$ is a
triangle in $\underline{\gr} E$ with $U, V\in \mathcal{C}_\lambda$,
then both $Q$ and $\Omega^{-1}U$ are in $\mathcal{C}_\lambda$,
where $\Omega^{-1}(U)$ is the first syzygy of $U$ by taking
minimal graded injective resolution of $U$. Hence
$\mathcal{C}_\lambda$ is a triangulated subcategory of
$\underline{\gr} E$. Now if $N$ is a direct summand of $M$ in
$\underline{\gr} E$, then $N$ is a direct summand of $M\oplus P$
in $\gr E$ for some finite dimensional projective module $P$.
Thus, we have $c(N)\leq c(M)$. Therefore $\mathcal{C}_\lambda$ is
thick.
\end{proof}

Let $D^b(proj E)$ be the full subcategory of
$D^b(\gr E)$ consisting of all bounded complexes of f.g.
graded projective $E$-modules. Then
$D^b(proj E)$ is equivalent to the triangulated subcategory
of $D^b(\gr E)$ generated by $E$. Since $E$ is a graded
Frobenius algebra, there is a well known equivalence of
triangulated categories \cite[Corollary 3.9(c)]{Be}
\begin{equation}\label{E4.4.1}\tag{E4.4.1}
\Phi:D^b(\gr E)/D^b(proj E)\longrightarrow \underline{\gr} E,
\end{equation}
where the functor $\Phi$ is defined in the following way.
Let $X^\cdot$ be an object in $D^b(\gr E)$. We write
$\overline{X^\cdot}$ for its image in $D^b(\gr E)/D^b(proj E)$.
Take a bounded above f.g. projective resolution
$P^\cdot$ of $X^\cdot$. Since $X^\cdot$ is bounded, we see
that there is an integer $n$ such that the cohomology
$H^iP^\cdot=0$ for all $i\leq -n$. Now, let $M$ be the
$n$th syzygy of the complex $P$. Note that in the triangulated category $D^b(\gr E)/D^b(proj E)$, the complex $P$ is isomorphic to its naive truncation $P^{\leq -n}$, which is isomorphic to $M[n-1]$. Hence
$\Phi(\overline{X^\cdot})=\Omega^{1-n}(M)$ \cite[Sect. 1.3]{Or}.

One can easily extend the notion of the complexity to
bounded complexes of modules over a finite dimensional algebra.
Let $\Lambda:=\Lambda_0\oplus \Lambda_1\oplus \cdots $ be a finite
dimensional graded algebra such that $\Lambda_0$ is a semisimple
algebra. Let $K^-(\gr \Lambda)$ be the homotopy
category of bounded above complexes of f.g. right
graded $\Lambda$-modules. A bounded above complex of graded
projective modules $P^\cdot\in K^-(\gr \Lambda)$ is called
{\it minimal} if ${\rm{im}} d^n$ is superfluous, namely,
${\rm{im}} d^n\subseteq P^n\Lambda_{\ge1}$ for all $n$, where
$\Lambda_{\ge1}=\Lambda_1\oplus\Lambda_2\oplus\cdots$. Note
that, for any bounded complex $X^\cdot$  of f.g.
right graded $\Lambda$-modules, there is a minimal complex of
graded projective modules $P^\cdot$ that is quasi-isomorphic
to $X^\cdot$. Such $P^\cdot$ is called a {\it minimal projective
resolution} of $X^\cdot$. It is known that minimal projective
resolution of a bounded complex of f.g. graded
$\Lambda$-module is unique up to isomorphism. We define the
{\it complexity} $c(X^\cdot)$ of $X^\cdot$ to be the least
integer $\lambda$ such that $\dim P^{-n}\leq a\, n^{\lambda-1}$
for almost all $n\ge0$. If two bounded complexes $X^\cdot$ and
$Y^\cdot$ are quasi-isomorphic, then their minimal projective
resolutions are isomorphic. Hence $c(X^\cdot)=c(Y^\cdot)$. Hence,
it makes sense to define the complexity of an object in
$D^b(\gr A)$. For $\lambda\ge0$, we denote by
$D^b_\lambda(\gr \Lambda)$ the full subcategory of
$D^b(\gr \Lambda)$ consisting of objects $X^\cdot$ such that
$c(X^\cdot)\leq \lambda$. One sees that $D_0^b(\gr \Lambda)$ is
equivalent to $D^b(proj \Lambda)$.

\begin{lemma}
\label{xxlem4.5}
Retain the above notation.
Then $D^b_\lambda(\gr \Lambda)$ is a thick triangulated
subcategory of $D^b(\gr \Lambda)$ for all $\lambda\ge 0$.
\end{lemma}

\begin{proof}
Take $X^\cdot, Y^\cdot\in D^b_\lambda(\gr \Lambda)$.
Let $f:X^\cdot \to Y^\cdot$ be a morphism in $D^b(\gr \Lambda)$,
and $X^\cdot \to Y^\cdot\to Z^\cdot\to X^\cdot[1]$ be a triangle.
Let $P^\cdot$ and $Q^\cdot$ be minimal projective resolutions
of $X^\cdot$ and $Y^\cdot$ respectively. The morphism $f$
induces a morphism $g:P^\cdot\to Q^\cdot$. Then $cone(g)\cong Z$
in $D^b(\gr \Lambda)$. Hence the minimal projective resolution
of $Z^\cdot$ is a direct summand of $cone(g)$. Therefore,
$c(Z^\cdot)\leq c(cone(g))\leq \lambda$, and hence
$D^b_\lambda(\gr \Lambda)$ is a triangulated subcategory
of $D^b(\gr \Lambda)$. It is clear that $D^b_\lambda(\gr \Lambda)$
is closed under taking direct summands.
\end{proof}

We can also extend the notion of the GK-dimension from modules to
bounded complexes. Let $S$ be a noetherian locally finite
graded algebra. Let $X^\cdot$ be a bounded complex
of right graded $S$-modules. Set
$H^\cdot(X^\cdot)=\bigoplus_{n\in{\mathbb Z}}H^n(X^\cdot)$, the total
cohomology of $X^\cdot$. Define the GK-dimension of
$X^\cdot$ as
$$\GKdim X^\cdot:=\GKdim H^\cdot(X^\cdot).$$
By definition, if two complexes $X^\cdot$ and $Y^\cdot$ are
quasi-isomorphic, then $\GKdim X^\cdot=\GKdim Y^\cdot$. Hence
it makes sense to say the GK-dimension of an object in the
bounded derived category. For $\lambda\ge0$, let
$D^b_{\GKgr_\lambda}(\gr S)$ be the full subcategory of
$D^b(\gr S)$ consisting of objects $X^\cdot$ such that
$\GKdim X^\cdot\leq \lambda$. One sees that $X^\cdot\in
D^b_{\GKgr_\lambda}(\gr S)$ if and only if $\GKdim
H^n(X^\cdot)\leq \lambda$ for all $n\in {\mathbb Z}$. Hence
we have the following.

\begin{lemma}
\label{xxlem4.6}
For every $\lambda \ge0$, $D^b_{\GKgr_\lambda}(\gr S)$ is
a thick triangulated subcategory of $D^b(\gr S)$.
\end{lemma}

We now go back to the $H$-action on a noetherian
connected graded AS regular algebra $R$.

\begin{setup}
\label{xxset4.7}
Throughout the rest of this section, we assume that $R$ is
a noetherian connected graded AS regular algebra which is
also a Koszul algebra. Let $B$ be the graded basic algebra
of $R\#H$. Recall that $\Bbbk$ is assumed to be algebraically
closed. Then $B_0\cong \Bbbk^{\oplus w}$ for some
$w>0$ [Lemma \ref{xxlem4.1}(4)]. Let
$E=\oplus_{n\ge0} \uExt_B^n(B_0,B_0)$ as given in \eqref{E4.0.1}.
\end{setup}

\begin{lemma}
\label{xxlem4.8}
Let $M$ be a right Koszul $B$-module. Then $\GKdim (M)=c(K(M))$.
\end{lemma}

\begin{proof}
Since $B$ is a Koszul algebra, $K(M)$ is a Koszul $E$-module,
moreover, $K(M)$ has a minimal projective resolution
\cite[Proposition 5.1]{GM}
$$\cdots\longrightarrow P^{-n}\longrightarrow\cdots
\longrightarrow P^{-1}\longrightarrow P^0\longrightarrow K(M)
\longrightarrow0$$
such that $P^{-n}\cong M_n^*\otimes_{E_0} E$, where $M^*_n=\Hom(M_n,\Bbbk)$.

Assume $c(K(M))=\lambda$. Then $\dim M_n\leq \dim P^{-n}\leq
an^{\lambda-1}$ for some fixed number $a>0$ and for all $n\gg 0$.
Hence $\GKdim(M)\leq \lambda$.

Conversely, assume $\GKdim(M)=\lambda$. By assumption $B$ is
the basic algebra of $R\#H$. Hence there is a finitely
generated graded progenerator $R\#H$-module $Q$,
where $Q:=\bigoplus_{i=1}^w Q_i$ as given in
\eqref{E4.0.2}, such that ${\underline{\End}}_{R\#H}(Q)\cong B$.
Write $F=-\otimes_{R\#H} Q$. Assume $N$ is a graded right
$R\#H$-module such that $F(N)\cong M$. Then
$\GKdim(N)=\GKdim(M)=\lambda$ by Remark \ref{xxrem1.5}(3).
Note that $N$ is also a graded
right $R$-module and by assumption $R$ is a noetherian connected
graded AS regular algebra. By \cite[Proposition 2.21]{ATV} and
\cite[Corollary 2.2]{StZ}, we have $\dim N_n\leq a n^{\lambda-1}$
for some fixed number $a>0$ and for all $n\gg 0$.
Since $M=F(N)$, there is an $a_1>0$ such that
$\dim M_n\leq a_1 \dim N_n\leq a_1 a n^{\lambda-1}$ for all
$n\gg 0$. Assume $\dim E=b$. Then
$\dim P^{-n}\leq b\dim M_n\leq ba_1 an^{\lambda-1}$. Hence
$c(K(M))\leq \lambda$.
\end{proof}

\begin{lemma} \label{xxlem4.9}
For every integer $\lambda\geq 0$,
the Koszul duality given in Proposition {\rm{\ref{xxpro4.3}}}
restricts to the following anti-equivalence of triangulated
categories,
$$D^b_{\GKgr_\lambda}(\gr B)\cong D^b_\lambda(\gr E).$$
\end{lemma}

\begin{proof}
Firstly, notice that if $M$ is a f.g. Koszul
graded $B$-module, then, by Proposition \ref{xxpro4.3},
$H^i(K(M))=0$ for $i\neq0$.
In particular, $K(B_0)\cong E$. Hence
we see that the thick triangulated subcategory of $D^b(\gr B)$
generated by $B_0$ is anti-equivalent to the thick triangulated
subcategory of $D^b(\gr E)$ generated by $E$. The former
is equivalent to $D^b_{\GKgr_0}(\gr B)$ and the latter  is
equivalent to $D^b(proj E)$ which is in turn equivalent to
$D^b_0(\gr E)$. So we have proven the assertion for $\lambda=0$.

Next we assume that $\lambda\ge1$. By the Koszul duality, we
only need to show the following:
\begin{enumerate}
\item[(a)]
If $X^\cdot\in D^b_{\GKgr_\lambda}(\gr B)$, then
$K(X^\cdot)\in D^b_\lambda(\gr E)$.
\item[(b)]
If $Y^\cdot\in D^b_\lambda(\gr E)$, then
$K^{-1}(Y^\cdot)\in D^b_{\GKgr_\lambda}(\gr B)$.
\end{enumerate}
Since $D^b_{\GKgr_\lambda}(\gr B)$ and $D^b_\lambda(\gr E)$
are thick triangulated subcategories, it suffices to show
that,
\begin{enumerate}
\item[(c)]
for every f.g. right graded $B$-modules $M$
with $\GKdim\leq\lambda$, we have
$K(M)\in D^b_\lambda(\gr E)$, and
\item[(d)]
for every f.g.
right graded $E$-modules $N$ such that
$c(N)\leq\lambda$, we have $K^{-1}(N)\in D^b_{\GKgr_\lambda}(\gr B)$.
\end{enumerate}
Let $M$ be a f.g. right $B$-module with $\GKdim M\leq\lambda$.
Since $B$ is a Koszul AS-regular algebra, it follows that
for sufficiently large $k$ the graded $B$-module
$M_{\geq k}(k)$, where
$M_{\ge k}=M_k\oplus M_{k+1}\oplus\cdots$, has linear
resolution (see \cite[Theorem 3.1]{Jo1}, note that
it is assumed that the algebra is connected graded
in \cite {Jo1}, however, the same proof holds in our case),
that is, $M_{\ge k}(k)$
is a Koszul module. Consider the exact sequence
$$0\to M_{\ge k}\to M\to M/M_{\ge k}\to 0.$$
We see that $\GKdim M_{\ge k}=\GKdim M\leq \lambda$
and
$$K(M/M_{\ge k})\to K(M)\to K(M_{\ge k})\to K(M/M_{\ge k})[1]$$
is a triangle in $D^b(\gr E)$. Since $M/M_{\ge k}$ is
finite dimensional, $K(M/M_{\ge k})\in D^b_0(\gr E)$.
Since $M_{\ge k}(k)$ is a Koszul module, $c(K(M_{\ge k}))=
\GKdim(M_{\ge k})\leq \lambda$ by Lemma \ref{xxlem4.8}.
Hence $K(M_{\ge k})\in D^b_{\lambda}(\gr E)$. Therefore
$K(M)\in D^b_{\lambda}(\gr E)$. This verifies claim
(c).

For $Y^\cdot\in D^b_{\lambda}(\gr E)$, let $P^\cdot$ be a
minimal projective resolution of $Y^\cdot$.
Then $\dim P^{-n}\leq a\, n^{\lambda-1}$ for $n\gg 0$.
Since $B_0$ is a direct sum of finitely many copies of $\Bbbk$,
so is $E_0$. Assume that $E_0=\Bbbk\oplus\cdots\oplus\Bbbk$
is the sum of $s$ copies of $\Bbbk$, and that $e_1,\dots,e_s$
are the corresponding primitive idempotents. Then
$E=\oplus_{i=1}^sEe_i$. Assume $r=\min\{\dim Ee_i|i=1,\dots,s\}$.
By the minimality of $P^\cdot$, we see that
$$\dim P^{-n}\ge r\sum_{j\in {\mathbb Z}}\dim
\Hom_{D^b(\gr E)}(Y^\cdot, E_0[n](j)).$$
Hence
$$\sum_{j\in {\mathbb Z}}\dim \Hom_{D^b(\gr E)}(Y^\cdot, E_0[n](j))\leq
\frac{a}{r}\, n^{\lambda-1}.$$
Given $i$ and $j$, in the following computations, note that
$K(B)\cong E_0$ in $D^b(\gr E)$,
\begin{eqnarray*}
  \dim H^iK^{-1}(Y^\cdot)_j&=&\dim \Hom_{D^b(\gr B)}(B[-i](-j),K^{-1}(Y^\cdot))\\
  &=&\dim \Hom_{D^b(\gr E)}(Y^\cdot,K(B[-i](-j)))\\
  &=& \dim \Hom_{D^b(\gr E)}(Y^\cdot,E_0[i+j](-j))\\
  &\leq& \frac{a}{r}(i+j)^{\lambda-1}.
\end{eqnarray*}
Hence $\GKdim H^iK^{-1}(Y^\cdot)\leq \lambda$ for all $i$.
Therefore, $K^{-1}(Y^\cdot)\in D^b_{\GKgr_\lambda}(\gr B)$.
Thus we verify claim (b) (and (d)).
\end{proof}

\begin{lemma}
\label{xxlem4.10}
There are equivalences of triangulated categories:
$$D^b_\lambda(\gr E)/D^b(proj E)\cong \mathcal{C}_\lambda,$$
for all $\lambda\ge0$.
\end{lemma}

\begin{proof}
By Lemma \ref{xxlem4.5}, $D^b_\lambda(\gr E)$ is a thick
triangulated subcategory of $D^b(\gr E)$ for each $\lambda\ge0$.
One sees that $D^b_\lambda(\gr E)/D^b(proj E)$ is a full
triangulated subcategory of $D^b(\gr E)/D^b(proj E)$ for
each $\lambda\ge1$. Consider the equivalence
$$\Phi:D^b(\gr E)/D^b(proj E)\longrightarrow \underline{\gr} E$$
as established by (\ref{E4.4.1}). By the definition of $\Phi$,
one easily sees that if $c(X^\cdot)\leq\lambda$ then
$c(\Phi(X^\cdot))\leq\lambda$. Hence $\Phi$ sends objects
in $D^b_\lambda(\gr E)/D^b(proj E)$ to objects in
$\mathcal{C}_\lambda$. If $N$ is a right graded $E$-module
such that $c(N)\leq \lambda$, then we may view $N$ as an
object in $D^b_\lambda(\gr E)/D^b(proj E)$, and we have
$\Phi(N)\cong N\in \mathcal{C}_\lambda$. Therefore the
restriction of $\Phi$ to $D^b_\lambda(\gr E)/D^b(proj E)$
is essential dense onto $\mathcal{C}_\lambda$, and hence
$D^b_\lambda(\gr E)/D^b(proj E)$ is equivalent to
$\mathcal{C}_\lambda$.
\end{proof}

We are ready to prove the main result of this section,
which may be viewed as a generalization of the classical
BGG correspondence \cite{BGG, Jo2, Mor1}.

\begin{theorem}
\label{xxthm4.11}
Let $B$ and $E$ be as in Setup {\rm{\ref{xxset4.7}}}. Then,
for each $\lambda\ge0$, there is an
anti-equivalence of triangulated categories
$$D^b(\qgr_\lambda B)\cong \underline{\gr}E/\mathcal{C}_\lambda.$$
\end{theorem}

\begin{proof} By the Koszul duality in Proposition \ref{xxpro4.3},
we have an anti-equivalence $K:D^b(\gr B)\longrightarrow D^b(\gr E)$.
Lemma \ref{xxlem4.9} and its proof show that $K$ restricts to
anti-equivalences $D^b_{\GKgr_\lambda}(\gr B)\cong D^b_\lambda(\gr E)$
for each $\lambda\ge0$. Hence we obtain anti-equivalences
$$D^b(\gr B)/D^b_{\GKgr_\lambda}(\gr B)\cong D^b(\gr E)/D^b_\lambda(\gr E)$$
for each $\lambda\ge0$. By the equivalence
$\Phi:D^b(\gr E)/D^b(proj E)\longrightarrow \underline{\gr} E$ and
Lemma \ref{xxlem4.10}, we have
$$D^b(\gr E)/D^b_\lambda(\gr E)\cong \underline{\gr}E/\mathcal{C}_\lambda.$$
On the other hand, by \cite[Theorem 3.2]{Mi} we have the equivalence of
triangulated categories
$$D^b(\gr B)/D^b_{\GKgr_\lambda}(\gr B)\cong D^b(\qgr_\lambda B),$$
since $\GKgr_\lambda B$ is a Serre subcategory of $\gr B$. Combine the
(anti-)equivalences above, we finally obtain the anti-equivalences as required.
\end{proof}

\begin{remark}
\label{xxrem4.12}
In the case when $\lambda=0$, we have $\mathcal{C}_0=0$ and
$\qgr_0 B=\qgr B$, the quotient category of $\gr B$ modulo
by the Serre subcategory of finite dimensional graded
$B$-modules. Then the anti-equivalence in the theorem above
has the form
$$D^b(\qgr B)\cong \underline{\gr} E$$
which is the classical BGG correspondence established in
\cite{BGG}, \cite[Theorem 3.1]{Jo2} and \cite[Theorem 5.3]{Mor1}.
\end{remark}

\begin{corollary}
\label{xxcor4.13}
Assume Setup {\rm{\ref{xxset4.7}}}.
Let $R$ be CM{\rm{$(2)$}} of $\GKdim $ $d\geq 2$.
If $R\#H\cong \text{End}_{R^H}(R)$, then there is a canonical
anti-equivalence of triangulated categories
$$D^b(\qgr_{d-2} R^H)\cong \underline{\gr} E/\mathcal{C}_{d-2}$$
where $E=\bigoplus_{n\geq 0} \Ext^n_{R\# H}(H,H)$.
\end{corollary}

\begin{proof}
Note that $R$ and $H$ are in the situation as in Theorem \ref{xxthm3.5}.

Assume that $R\#H\cong \text{End}_{R^H}(R)$, we will show the category
equivalence. Take $B$ to be the basic algebra of $R\#H$. Then
$R\#H$ and $B$ are graded Morita equivalence. Since $R$ is a
Koszul algebra, $E$ and $E':=\bigoplus_{n\geq 0} \Ext^n_{B}(B_0,B_0)$
are also graded Morita equivalent [Lemma \ref{xxlem4.1}(2)].
Since the graded Morita equivalence preserves the GK-dimension, we
have that $\GKgr_\lambda(R\#H)$ and $\GKgr_\lambda B$
are equivalent. As a consequence, $D^b(\qgr_\lambda B)$ is equivalent to
$D^b(\qgr_\lambda R\#H)$. Combining this with Theorems \ref{xxthm3.5}
and \ref{xxthm4.11}, we have
$$D^b(\qgr_{d-2} R^H)\cong
D^b(\qgr_{d-2} R\#H) \cong D^b(\qgr_{d-2} B)
\cong\underline{\gr} E'/\mathcal{C}_{d-2}
\cong\underline{\gr} E/\mathcal{C}_{d-2}.$$
\end{proof}

\begin{proof}[Proof of Theorem \ref{xxthm0.7}]. (1) is Theorem \ref{xxthm4.11}. 
(2) follows from Theorems \ref{xxthm4.11} and \ref{xxthm3.7}.
\end{proof}

\section{Proof of Theorem \ref{xxthm0.5}}
\label{xxsec5}

In this section let $\Bbbk$ be a field and ${\rm{char}}\;
\Bbbk \nmid 2n$. We consider some finite group actions on
the $(-1)$-skew polynomial algebra $\Bbbk_{-1}[x_1,\dots,x_n]$
and prove Theorem \ref{xxthm0.5}. We need a few lemmas.

Let $R$ be a graded domain. Let $G$ be a finite group and $H$ its dual group algebra $(\Bbbk G)^{\circ}$
acting on $R$ homogeneously. Assume that the $H$-action on $R$ is
{\it inner faithful}. If $R$ is generated in degree 1, then this is equivalent to that
$H$-action on $R_1$ is inner faithful. In this setting, $R$ is also
${\mathbb Z}\times G$-graded. Forgetting the ${\mathbb Z}$-grading,
$R$ is $G$-graded and $R$ is not $G'$-graded for any proper subgroups
$G'\subsetneq G$. Write $R=\oplus_{g\in G} R_g$. Part (4) of the next lemma
is due to \cite{CKWZ2}. Let $I$ be the ideal $(e)$ in $B:=R\# H$
where $e=1\# \inth$ and $\inth$ is the integral of $H$ such that
$\epsilon(\inth)=1$.

\begin{lemma}
\label{xxlem5.1}
Let $R$ be a connected graded domain.
Let $H:=(\Bbbk G)^{\circ}$ act on $R$ inner faithfully and write
$R=\oplus_{g\in G} R_g$.
\begin{enumerate}
\item[(1)]
For every $g\in G$, $R_g\neq 0$.
\item[(2)]
$R$ is a free $H$-module.
\item[(3)]
If $f\in \bigcap_{g\in G}R R_g$, then $f\# 1\in I$.
\item[(4)]
\cite[Lemma 4.8(2)]{CKWZ2}
Suppose $f_1, \cdots, f_m\in R$ are $G$-homogeneous elements in $R$
such that  $\{h_w\mid w=1, \cdots, m\}=G$ where $$h_w:=\deg_G(f_w)\cdots \deg_{G}(f_m).$$
Then $f_1\cdots f_m \# 1\in I$.
\item[(5)]
Suppose $d:=\GKdim R<\infty$. Then $\Pty(R,H)\geq 1$.
As a consequence, $R\# H$ is prime and the algebra map
$$R\# H\to \End_{R^H}(R)$$
is injective.
\item[(6)]
$\Pty(R,H)\geq d-\GKdim R/(\bigcap_{g\in G}R R_g)\geq
d-\max\{\GKdim R/RR_g\mid g\in G\}.$
\item[(7)]
If there is an integer $s$ such that $RR_g\supseteq R_{\geq s}$
for all $g\in G$, then $\Pty(R,H)=d$.
\end{enumerate}
\end{lemma}

\begin{proof} (1) Let $X=\{g\in G \mid R_g\neq 0\}$. The Hopf ideal 
associated with $X$ is $J=\{f\in(\Bbbk G)^\circ|f(X)=0\}$. Since $R$ is a
domain, $X$ is a subsemigroup of $G$. Since $X$ is finite, it is a
group. By the inner faithfulness $J=0$, and hence $X=G$.

(2) We may assume $R\neq \Bbbk$. Since $R$ is a
connected graded domain, $R$ is infinite dimensional. As a
consequence, $R^H$ is infinite dimensional. So each $R_g$ is infinite
dimensional. So $R$ is a free $H$-module of infinite rank.

(3)
Let $\{g_1, \cdots, g_m\}$ be the set $G$ where
$g_1$ is the unit of $G$. The dual basis $\{p_{g_1},\cdots, p_{g_m}\}$
is the complete set of orthogonal idempotents of $H$ that forms
a $\Bbbk$-linear basis of $H$. The integral of $H$ is $p_{g_1}$ and
$\sum_{i=1}^m p_{g_{i}}$ is the algebra unit of $H$. The coproduct
of $H$ is determined by
$$\Delta(p_{g})=\sum_{h\in G} p_{h}\otimes p_{h^{-1}g}$$
for all $g\in G$.
Let $f\in R$ be a $G$-homogeneous element of degree $g\in G$.
Then
$$e( f\# 1)=(1\# p_{g_1}) ( f\# 1)=\sum_{h} p_{h}(f)\# p_{h^{-1}}
=p_g(f)\# p_{g^{-1}}=f\# p_{g^{-1}}.$$
This implies that
$$f\# p_{g^{-1}}\in I$$ for all $f\in R_g$.
Since $I$ is a left $R$-module, we have
$$RR_{g} \# p_{g^{-1}}\subseteq I$$
for all $g\in G$. If $f\in \bigcap_{g\in G} RR_g$, then $f\# p_{g^{-1}}\in I$
for all $g\in G$. Then $f\# 1=\sum_{g\in G} f\# p_{g^{-1}}\in I$.

(4) Let $f=f_1\cdots f_m$. Then $f_{w}\cdots f_{m}
\in R_{h_w}$ and $f\in RR_{h_w}$ for all $w$.
By hypothesis $\{h_w\}^{m}_{w=1}$ equals $G$. Then $f\in \bigcap_{g\in G}
RR_g$. The assertion follows by part (3).

(5) Assume $|G|=m$ and write $G=\{g_1,\dots,g_{m-1},g_m=e\}$. By Part (1), 
we may choose $0\neq f_m\in R_e$ and $0\neq f_i\in R_{g_ig^{-1}_{i+1}}$ 
for $i=1,\dots,m-1$. Then $f_1,\dots,f_m$ satisfy the condition in part (4). 
Since $R$ is a domain, $f:=f_1\cdots f_m\neq 0$. By part (4), $f\#1\in I$.
Then $B/(I)$ has GK-dimension no
more than the $\GKdim R-1$, or equivalently, $\Pty(R,H)\geq 1$.
The assertion follows by Lemma \ref{xxlem3.10}(2,4).

(6) This follows from Lemma \ref{xxlem5.2} below and
the fact that the image of $R\to R\# H/(I)$, denoted by
$R'$, is a quotient of $R/(\bigcap_{g\in G} RR_g)$, by part (3),
and that there is an embedding $R/(\bigcap_{g\in G} RR_g)\to
\bigoplus_{g\in G} R/(RR_g)$.

(7) This is an immediate consequence of (6).
\end{proof}

Let $T$ be a subalgebra of $R$ such that $R_T$ and $_TR$ are f.g..

\begin{lemma}
\label{xxlem5.2}
Let $R'$ be the image of the map $R\to B\to B/I$ and $T'$ be the
image of $T$ in $R'$. Then $\GKdim T'=\GKdim R'=\GKdim B/I.$
\end{lemma}

\begin{proof} This follows from the fact that $B/I$ is f.g.
 over $R'$.
\end{proof}

Aiming for a proof of Theorem \ref{xxthm0.5}, we now let $R$ be the
$(-1)$-skew polynomial ring $\Bbbk_{-1}[x_1,\cdots,x_n]$. Let $H$ be the
group algebra $\Bbbk W$ where $W$ is generated by the permutation
$\sigma: x_i\to x_{i+1}, x_{n}\to x_1$
for all $1\leq i\leq n-1$. Theorem \ref{xxthm0.5} asserts that
$R\# \Bbbk W\cong \End_{R^W}(R)$.

Let $\xi$ be a fixed primitive $n$th root of unity.
We might assume that $\xi\in \Bbbk$ since $\Pty(R,H)$
is preserved by field extensions. Let $y_j=
\sum_{i=1}^{n} \xi^{ij} x_i$, for $j=0,1,\dots,n-1$. Then $\{y_0,\cdots, y_{n-1}\}$
is a basis of $R_1$. It is clear that $\sigma(y_j)=\xi^{-j} y_j$.
Let $T$ be the central subalgebra of $R$
generated by $\{x_i^2\}_{i=1}^{n}$ and let $Y_j=\sum_{i=1}^{n} \xi^{ij} x_i^2$ for $j=0,1,\dots,n-1$.
Then $\{Y_i\}_{i=0}^{n-1}$ is a basis of the vector space
$\oplus_{i=1}^n \Bbbk x_i^2$. It is clear that $R$ is a f.g. module
over the central subalgebra $T$.

The group algebra $\Bbbk W$ can be viewed as a $(\Bbbk G)^\circ$ where $G$ is the
character group of $W$. Note that $G\cong W$ as abstract groups. Let
$g$ be a generator of $G$ such that $\deg_G (y_1)=g$. Then $\deg_G (y_i)=g^i$
for all $0\leq i\leq n-1$.

Let $U_{n}=\{(0,1), (1,2), \cdots, (\lfloor \frac{n}{2} \rfloor-1,
\lfloor \frac{n}{2} \rfloor)\}$. For any integer $i$,
let $\overline{i}$ be the integer between 0 and $n-1$ such that
$\overline{i}=i\mod n$. The following lemma is easy to check.

\begin{lemma}
\label{xxlem5.3} Retain the above notation for $R$. Let $i,j,k$ be integers
{\rm{(}}between $0$ and $n-1$ modulo $n${\rm{)}}.
\begin{enumerate}
\item[(1)]
$y_iy_j+y_jy_i=2Y_{i+j}=2Y_{\overline{i+j}}$.
\item[(2)]
$y_k y_j y_i=-y_k y_i y_j+(y_iy_j+y_jy_i)y_k$.
\item[(3)]
$R'/(Y_i\mid i=0,\cdots,n-1)$ is finite dimensional.
\item[(4)]
If $n$ is even and $i\geq \frac{n}{2}$, then $y_i^2=y_{j}^2$ where $j=i-\frac{n}{2}$. 
\item[(5)]
Assume $n$ is even.
If $i<j$ and $(i,j)\notin U_n$
$$y_j y_i=\begin{cases}
-y_iy_j+2 y_s^2 & i+j={\text{even and $<n$}}, \; s=\frac{i+j}{2},\\
-y_iy_j+2 y_s^2 & i+j={\text{even and $\geq n$}}, \; s=\frac{i+j-n}{2},\\
-y_iy_j+(y_s y_{s+1}+y_{s+1}y_s) & i+j={\text{odd and $<n$}}, \; s=\frac{i+j-1}{2},\\
-y_iy_j+(y_s y_{s+1}+y_{s+1}y_s) & i+j={\text{odd and $>n$}}, \; s=\frac{i+j-1-n}{2}.
\end{cases}$$
\item[(6)]
Suppose $n=2^d$.
For each $s$, $(y_sy_{s-1})^n=0$ in $R'$, and
hence in $R'':=R'/(y_i^2\mid i=0,\cdots, \frac{n}{2})$.
\item[(7)]
Suppose $n=2^d$. Then
$R''$ is finite dimensional.
\item[(8)]
Suppose $n=2^d$.
Then $y_s^{2(d+1)}=0$ in $R'$ for all $s$.
\end{enumerate}
\end{lemma}

\begin{proof}
(1) By a direct computation.

(2) This follows from part (1) and the fact $Y_{\overline{i+j}}$
is central in $R$.

(3) This follows from Lemma \ref{xxlem5.2} and the fact
$T/(Y_i\mid i=0,\cdots,n-1)=k$.

(4) Follows from the computations in the proof of part (1).

From now on we assume that $n$ is even.

(5) It follows from part (1).

(6) Since $\deg_G(y_s y_{s-1})$ is an odd power of $g$, it is a 
generator of $G$, and hence the assertion follows from Lemma 
\ref{xxlem5.1}(4).

(7) We use Diamond lemma.
Fix an order $y_0 < y_1 < \cdots < y_{n-2}< y_{n-1}$ in $R''$, and
extend it to the lexicographical order on all monomials.

By part (5), for any $j>i$, $y_j y_i$ can be written
as a linear combination of lower terms. By part (4),
$y_i^2$ when $i\geq \frac{n}{2}$ can be expressed as 
a lower term. By part (2), if $k>j>i$, $y_ky_jy_i$
can be written as a linear combination of lower terms.

Let $X$ be a reduced monomial basis of $R''$, which means 
that if $X'$ is another monomial basis of $R''$, then, 
for any monomial $f\in X$, there is a monomial $f'\in X'$ 
with the same length as $f$ such that either $f=f'$ or $f$ 
is lower than $f'$.
Let $0\neq f=y_{i_1}y_{i_2}\cdots y_{i_h}\in X$.
Then $i_\alpha \neq i_{\alpha+1}$ since $y_i^2=0$ in $R''$ for all $i$.
Since there are only $n$ possible $i_\alpha$,
the maximal length of increasing $i_\alpha$ is at most $n$. 
If $i_\alpha> i_{\alpha+1}$, then, only the third or forth case 
of part (5) may occur, and hence $i_\alpha\leq \frac{n}{2}$
and $i_{\alpha+1}=i_\alpha-1$, namely, $(i_{\alpha+1},i_{\alpha})\in U_n$.

We next prove that there are at most $\frac{1}{2} n^2$
pairs $(i_{\alpha+1},i_\alpha)\in U_n$ appeared in $f$. 
Suppose we have two pairs
$(i_{\alpha+1},i_\alpha)$ and
$(i_{\beta+1},i_\beta)$ in $U_n$ with $\alpha<\beta$.
By part (2), $\beta\neq \alpha+1$.
If $\beta> \alpha+2$, then $i_s$ are strictly increasing
between $i_{\alpha+1}$ and $i_\beta$ (up to taking $\beta$ minimal). 
So $i_\beta\geq i_{\alpha}+1$. If $\beta=\alpha+2$,
then $i_\beta\geq i_{\alpha}$. Now let $Y_{i}=y_{i}y_{i-1}$ 
for all $i\leq \frac{n}{2}$. By part (6), $Y_i^n=(y_{i}y_{i-1})^{n}=0$
in $R''$. The above argument shows that $f$ is of the form
$$y_{i_1}^{d_1} Y_{i_2}^{e_1} y_{i_3}^{d_2} Y_{i_4}^{e_2}
\cdots  y_{i_{2w-1}}^{d_w} Y_{i_{2w}}^{e_w}$$
where 
$$i_1<i_2\leq i_3< i_4\leq i_5\cdots$$
and where $0\leq d_s\leq 1$ and $0\leq e_s\leq n-1$. 
By definition, $1\leq i_2<i_4<\cdots \leq \frac{n}{2}$.
Therefore there are at most $\frac{n}{2} \cdot n$ pairs 
of $(i_{\alpha+1},i_\alpha)\in U_n$ appeared in $f$. 
Consequently, the degree of $f$ is uniformly bounded. 
This means that $X$ is finite and $R''$ has finite dimension.

(8) By part (4), we may assume $s<\frac{n}{2}$. If $s$ is odd
(which is coprime to $n$), then, by Lemma \ref{xxlem5.1}(4),
$y_s^{n}=0$ in $R'$.

Now we consider general $s$. By Lemma \ref{xxlem5.1}(3), it
suffices to show that $y_s^{2(d+1)}\in RR_{g^i}$ for any
$i=1,\cdots, n-1$. Note that $y_s^{2(d+1)}\in R=RR_{g^0}$ is clear.
Fix $i$ and write $\equiv$ to be $=\mod RR_{g^i}$. We start
with $y_i\equiv 0$. Then
\begin{equation}
\label{E5.3.1}\tag{E5.3.1}
2y_s^2=y_{i}y_{2s-i}+y_{2s-i}y_{i}\equiv y_{i}y_{2s-i}.
\end{equation}
We prove by induction that, for each $w$,
\begin{equation}
\label{E5.3.2}\tag{E5.3.2}
(2y_s)^{2w}\equiv y_{j_1} y_{j_2}\cdots y_{j_w} y_{i_w}
\cdots y_{i_2} y_{i_1}
\end{equation}
where
$$i_\alpha=2^\alpha s - 2^{\alpha-1} i, \quad {\text{and}}\quad
j_\alpha=-2^{\alpha} s + 2^{\alpha-1} i +2s.$$
When $w=1$, this is \eqref{E5.3.1}. Suppose \eqref{E5.3.2}
holds for $w-1\geq 1$. Then
$$\begin{aligned}
(2y_s)^{2w}&=(2y_s^2) (2y_s)^{2(w-1)}
\equiv (2y_s^2) y_{j_1} y_{j_2}\cdots y_{j_{w-1}} y_{i_{w-1}}
\cdots y_{i_2} y_{i_1}\\
&\equiv y_{j_1} y_{j_2}\cdots y_{j_{w-1}} (2y_s^2) y_{i_{w-1}}
\cdots y_{i_2} y_{i_1}\\
&\equiv  y_{j_1} y_{j_2}\cdots y_{j_{w-1}} (y_{i_w} y_{j_w}+
y_{j_w}y_{i_w})y_{i_{w-1} \cdots y_{i_2} y_{i_1}}\\
&\equiv  y_{j_1} y_{j_2}\cdots y_{j_{w-1}}
(y_{j_w}y_{i_w})y_{i_{w-1}} \cdots y_{i_2} y_{i_1}
\end{aligned}
$$
since $y_s^2$ is central in $R$ and $y_{j_w}y_{i_{w-1}} \cdots y_{i_2} y_{i_1}\in R_{g^i}$.
Therefore \eqref{E5.3.1} follows from the induction.
Now take $w=d$, we have
$$\deg_G(y_{i_d}\cdots y_{i_1})=g^{(2^{d+1}-2)s-(2^d-1)i}=g^{-2s+i}$$
which implies that
$\deg_G(y_s^2y_{i_d}\cdots y_{i_1})=g^{(2^{d+1}-2)s-(2^d-1)i}=g^{i}$
and
$$
\begin{aligned}
(2y_s)^{2(d+1)}&\equiv (2 y_s^2) y_{j_1} y_{j_2}\cdots y_{j_{d-1}}
y_{j_d}y_{i_d}y_{i_{d-1}} \cdots y_{i_2} y_{i_1}\\
&\equiv 2y_{j_1} y_{j_2}\cdots y_{j_{d-1}}
y_{j_d} (y_s^2 y_{i_d}y_{i_{d-1}} \cdots y_{i_2} y_{i_1})
\equiv 0.
\end{aligned}
$$
Since ${\text{char}}\; \Bbbk\nmid 2n$, $y_s^{2(d+1)}\in RR_{g^i}$ for all
$i$ as desired.
\end{proof}

\begin{proposition}
\label{xxpro5.4}
Let $R$ be the $(-1)$-skew polynomial ring $\Bbbk_{-1}[x_1,\cdots,x_n]$
for $n\geq 2$. Let $H$ be the group algebra $\Bbbk W$ where $W$ is
generated by the permutation $\sigma: x_i\to x_{i+1}, x_{n}\to x_1$
for all $1\leq i\leq n-1$.
Suppose $n=2^d$ for some $d\geq 1$.
\begin{enumerate}
\item[(1)]
$\Pty(R,H)=\GKdim R=n\geq 2$.
\item[(2)]
$\qgr_0 R\# H\cong \qgr_0 R^H$.
\item[(3)]
$R^H$ has graded isolated singularities in the sense of
\cite[Definition 2.2]{Ue}.
\end{enumerate}
\end{proposition}

\begin{proof} (1) By Lemma \ref{xxlem5.2}, it suffices
to show that $\GKdim R'=0$. Let ${\mathfrak p}$ be a
minimal prime ideal of $R'$ and let $S=R'/{\mathfrak p}$.
For any $s$, $y_s^2$ is normal in $R'$. By Lemma \ref{xxlem5.3}(8),
$(y_s^2)^{(d+1)}=0$ in $R'$. Since $S=R'/{\mathfrak p}$
is prime, $y_s^2=0$ in $S$. By Lemma \ref{xxlem5.3}(7),
$S$ is finite dimensional (and hence $S\cong \Bbbk$). 
Note that every noetherian algebra has finitely many minimal 
prime ideals \cite[Theorem 3.4]{GW}. As $R'$ is semiprime, 
it can be embedded into the product $\bigoplus_{\mathfrak p}R'/{\mathfrak p}$ 
where ${\mathfrak p}$ runs over all the minimal prime ideals of $R'$. 
Therefore $R'$ is finite dimensional and $\GKdim R'=0$ as desired.

(2) This is a consequence of Theorem \ref{xxthm0.6}.

(3) This follows from the definition and part (2).
\end{proof}

\begin{remark}
\label{xxrem5.5} We have some comments concerning
the above computation.
\begin{enumerate}
\item[(1)]
It is a little bit surprising that if $n=2^d$, then $\Pty(R,H)=
\GKdim R$. However conjecture \ref{xxcon0.9} would
give an explanation of this phenomena.
\item[(2)]
We conjecture that when $n\neq 2^d$, then $\Pty(R,H)<\GKdim R$
and $R^H$ does not have graded isolated singularities. It would be
nice if there is a formula for $\dim_{sing}(R^H)$ in terms of
$n$.
\end{enumerate}
\end{remark}

In practice, we assume that $H$ acts on $A$ inner faithfully.
When $R$ is an AS regular domain, we conjecture that $\Pty(R,H)\geq 1$,
see Lemma \ref{xxlem3.10}.
By Theorem \ref{xxthm3.5}, when $\Pty(R,H)\geq 2$, then the Auslander
theorem holds. By Lemma \ref{xxlem2.3}, when $\Pty(R, H)=\alpha$
(or more generally $\Pty(R,H)\geq \alpha$),
then $\qgr_{d-\alpha} R\# H\cong \qgr_{d-\alpha} R^H$ where
$d=\GKdim R$. Since $R\# H$ is regular, this equivalence gives some
bounds on the dimension of singular locus of $R^H$, see \eqref{E0.6.1}. 
Therefore $\Pty(R,H)$ relates several other properties and invariants
of the $H$-action on $R$.

The definition of reflection number $\Rpf(R,G)$ is given
in Definition \ref{xxdef0.8}.
Recall that $G$ is called {\it conventionally small}
if $\Rpf(R,G)\geq 2$. If the
Conjecture \ref{xxcon0.9} holds, then Auslander theorem holds
for classical small groups by Theorem \ref{xxthm3.5}.

\begin{remark}
\label{xxrem5.6}
\begin{enumerate}
\item[(1)]
By definition, $\Rpf(R,G')\geq \Rpf(R,G)$ if
$G'$ is a subgroup of $G$.
\item[(2)]
Suggested by part (1), we conjecture that
$\Pty(R,H')\geq \Pty(R,H)$ if $H'$ is a
Hopf subalgebra of $H$ (assuming $H'$-action
on $R$ is still inner faithful).
\item[(3)]
It is not clear to us if there is any
relationship between $\dim_{sing}(R^{H'})$
and $\dim_{sing}(R^H)$. Suggested by part (2),
we would guess that $\dim_{sing}(R^{H'})\leq
\dim_{sing}(R^H)$, but we don't have any
result to support the claim.
\end{enumerate}
\end{remark}

We can compute some bounds for $\Pty(R,G)$ and $\Rpf(R,G)$ in
the following situation.

\begin{theorem}
\label{xxthm5.7}
Let $R$ be the $(-1)$-skew polynomial ring $\Bbbk_{-1}[x_1,\cdots,x_n]$
for $n\geq 2$. Let $H$ be the group algebra $\Bbbk W$ where $W$ is
generated by the permutation $\sigma: x_i\to x_{i+1}, x_{n}\to x_1$
for all $1\leq i\leq n-1$.
\begin{enumerate}
\item[(1)]
If $n=2^d$, then $\Pty(R,W)=n$.
\item[(2)]
$\Rpf(R, W)=\begin{cases} n, & n=2^d\\
n(1-\frac{1}{p}), & {\text{$p>1$ is the minimal odd prime factor of $n$}}.
\end{cases}$
\item[(3)]
$\Pty(R, W)\geq \begin{cases}
\phi(n), & 4\nmid n\\
\frac{\phi(n)}{2}, & 4|n,
\end{cases}$\\
where $\phi(n):=n \prod_{{\text{all primes $p\mid n$}}}
(1-\frac{1}{p})$ is the Euler's totient function.
\item[(4)]
If $n=p^d$ where $p$ is a prime number, then
$$\Pty(R,W) \geq \Rpf(R,W) =
\begin{cases} n,& p=2,\\
n(1-\frac{1}{p}), & p>2.
\end{cases}$$
As a consequence,
Conjecture {\rm{\ref{xxcon0.9}}} holds in this case.
\end{enumerate}
\end{theorem}

\begin{proof}
(1) This is Proposition \ref{xxpro5.4}.

(2) Let $\sigma$ be the cycle permutation sending
$x_i\to x_{i+1}$ for all $i<n$ and sending
$x_n$ to $x_1$. Let $g$ be $\sigma^i$ for some
$i<n$. Then $g$ is a product of $a$-many of
disjoint cycles of length $b$ where $ab=n$.
Let $o(g)$ be the order of the pole of $Tr(g,t)$ at
$t=1$. By \cite[Lemma 1.7]{KKZ2}, $o(g)$ is equal to the
number of odd cycles. (The authors of \cite{KKZ2}
assume that ${\rm{char}}\; \Bbbk=0$, however,
\cite[Lemma 1.7]{KKZ2} holds when ${\rm{char}}\; \Bbbk\neq 2$.)
Hence $o(g)=a=n/b$ when $b$
is odd. If $n$ is of the form $2^d$, then $o(g)=0$
and $\Rpf(R,G)=n$. If $n$ contains some odd factor,
then the maximal possible $o(g)$ attains when $b$
is the smallest odd prime factor of $n$. Therefore
$\Rpf(R,G)=n-n/p$ where $p$ is the
minimal odd prime factor of $n$. The assertion follows.

(3) As before, let $R'$
be the image of $R$ in $B/I$ and let $T$ be the subalgebra
of $R'$ generated by $\{x_i^2\}_{i=1}^n$. Then $T$ is commutative
and $\GKdim T=\GKdim R'=\GKdim B/I$. It suffice to show that
\[\GKdim T/{\mathfrak p}\leq n-
\begin{cases}
\phi(n), & 4\nmid n,\\
\frac{\phi(n)}{2}, & 4|n,
\end{cases}\]
for all minimal prime ${\mathfrak p}$ of $T$.

Let $\xi$ be a primitive $n$th root of unity.
For any $0\leq i\leq n-1$, let $y_i=\sum_{j=1}^{n} \xi^{ij} x_j$
and $Y_i=\sum_{j=1}^{n} \xi^{ij} x_j^2$. Since $x_i x_j+x_jx_i=0$
for all $i\neq j$, we have $y_i^{2}=Y_{2i}$. Since $x_i^2$ are
central in $R$ for all $i$, $Y_{i}$ are all central in $R$ for
all $i$. Note that  the commutative ring $T/{\mathfrak p}$ is generated by
$\{x_i^2\}_{i=1}^{n}$, hence by $\{Y_i\}_{i=1}^n$.

Let $N_n$ denote the cardinality of
$\{Y_{2i}\mid \gcd(i,n)=1, i=0, 1, \cdots, n-1\}$. It is easily checked that $N_n$ is
 $\phi(n)$ when $4\nmid n$; and $\frac{1}{2} \phi(n)$
when $4|n$. We give some details in the case
when $n=2m$ where $m>1$ is an odd integer. It suffices to show that
$Y_{2i_0}\neq Y_{2i_1}$ if $\gcd (i_0, n)=\gcd (i_1, n)=1$
and $0\leq i_0< i_1\leq n-1$.
If $Y_{2i_0}=Y_{2i_1}$, then $n\mid 2(i_1-i_0)$. This implies that
$m\mid (i_1-i_0)$. Since both $i_1$ and $i_0$
are odd (as $\gcd (i_0, n)=\gcd (i_1, n)=1$),
$2m \mid (i_1-i_0)$, which is impossible.
So $N_n=\phi(n)$ in this case.

Suppose $\gcd(n,i_0)=1$. By Lemma \ref{xxlem5.1}, $y_{i_0}^{n} =0$ in $R'$
(by taking $f_i=y_{i_0}$ for all $i$). Then
$Y_{2i_0}^n=y_{i_0}^{2n}=0$. Since $T/{\mathfrak p}$ is prime,
$Y_{2i_0}=0$ in $T/{\mathfrak p}$.
Therefore $T/{\mathfrak p}$ is generated by
no more than $(n-N_n)$ elements. This implies that
$\GKdim T/{\mathfrak p}\leq n-N_n$ and the assertion follows.

(4) This is an immediate consequence of parts (1,2).
\end{proof}

Now we are ready to prove Theorem \ref{xxthm0.5}.

\begin{proof}[Proof of Theorem \ref{xxthm0.5}]
Note that $G=W$. By Theorem \ref{xxthm5.7}(1,3), $\Pty(R,W)\geq 2$.
It is well-known that skew polynomial rings are CM. As $G$ is a 
subgroup of $Aut_{gr}(R)$, the $\Bbbk G$-action on $R$ is 
clearly inner faithful. Then the final assertion follows 
from Theorem \ref{xxthm0.3}.
\end{proof}

Let $S_n$ act on $\{x_1,\cdots,x_n\}$ as permutations. This
action extends to an algebraic action on both commutative
polynomial ring $\Bbbk[x_1,\cdots,x_n]$
and the skew polynomial ring $\Bbbk_{-1}[x_1,\cdots,x_n]$.
We define the pertinency for $S_n$ and
its subgroups in these cases.

\begin{definition}
\label{xxdef5.8} Suppose ${\rm{char}}\; \Bbbk=0$.
Let $G$ be a subgroup of $S_n$ that
acts naturally on $\Bbbk[x_1,\cdots,x_n]$ and $\Bbbk_{-1}[x_1,\cdots,
x_n]$ respectively. Define the {\it positive and negative
pertinency} of $G$ to be
$$\Pty_{+}(G)=\Pty(\Bbbk[x_1,\cdots,x_n],G)
\quad {\text{and}}\quad
\Pty_{-}(G)=\Pty(\Bbbk_{-1}[x_1,\cdots,x_n],G)
$$
respectively.
\end{definition}

\begin{question}
\label{xxque5.9}
Is there a combinatorial algorithm to compute
$\Pty_{+}(G)$ and $\Pty_{-}(G)$ for all $G\subseteq S_n$?
\end{question}

The second part of the above question concerning
$\Pty_{-}(G)$ was suggested by Ellen Kirkman.
See \cite[Section 7]{BHZ} for more comments.

\subsection*{Acknowledgments}
The authors thank the referee for a very careful reading and 
useful comments. Many thanks to Kenneth Chan,
Andrew Conner, Jason Gaddis, Frank Moore,
Fredy Van Oystaeyen, Chelsea Walton, Yinhuo Zhang for
sharing their (un-published) notes and interesting ideas;
special thanks to Ellen Kirkman for many useful
conversations on the subject, for suggesting
a question which leads to Theorem \ref{xxthm0.5} and
Question \ref{xxque5.9},
for sharing her ideas concerning Lemma \ref{xxlem3.10},
and for pointing out a mistake of Theorem \ref{xxthm5.7}
in an earlier version.
Both Y.-H. Bao and J.-W. He were supported by NSFC
(No. 11571239, 11401001,11671351)
and J.J. Zhang by the US National Science
Foundation (grant Nos. DMS 1402863 and 
DMS 1700825).



\end{document}